\documentclass[11pt]{amsart}
\usepackage{amsmath,amssymb,amsthm,multicol,mathtools,hyperref,dsfont,pinlabel,enumitem,wasysym}
\usepackage[normalem]{ulem}
\usepackage[usenames,dvipsnames]{color}
\usepackage{mathrsfs}

\newcommand\lb{\!\!\left\bracevert\!}
\newcommand\rb{\!\right\bracevert\!\!}
\newcommand\cut{\setminus\!\setminus}
\newcommand\wt{\widetilde}
\newcommand\wh{\widehat}
\newcommand\gem{\scalebox{.75}{${\gemini}$}}
\newcommand\ari{\scalebox{.75}{${\aries}$}}

\newcommand\Z{\mathbb{Z}}

\newcommand\R{\mathbb{R}}
\newcommand\Gs{\mathcal{G}}

\newcommand\lla{\left\langle}
\newcommand\rra{\right\rangle}

\newcommand\bbm{\begin{bmatrix}}
\newcommand\ebm{\end{bmatrix}}
\newcommand\red[1]{\color{red}#1\color{black}}
\newcommand\FG[1]{\color{ForestGreen}#1\color{black}}
\newcommand\violet[1]{\color{Violet}#1\color{black}}
\newcommand\Navy[1]{\color{NavyBlue}#1\color{black}}

\newcommand\white[1]{\color{white}#1\color{black}}
\newcommand\brown[1]{\color{Brown}#1\color{black}}
\newcommand\sepia[1]{\color{Sepia}#1\color{black}}

\newcommand\Cyan[1]{\color{Cyan}#1\color{black}}
\newcommand\Orange[1]{\color{BurntOrange}#1\color{black}}
\newcommand\Gray[1]{\color{Gray}#1\color{black}}

\newcommand\inter[1]{\overset{_\circ}{\nu}#1}
\newcommand\PosCrScr{\raisebox{-1pt}{\includegraphics[height=7pt]{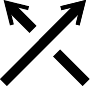}}}
\newcommand\NegCrScr{\raisebox{-1pt}{\includegraphics[height=7pt]{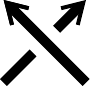}}}

\newcommand\PosCrRB{\raisebox{-2pt}{\includegraphics[height=11pt]{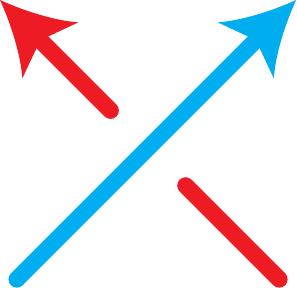}}}
\newcommand\NegCrRB{\raisebox{-2pt}{\includegraphics[height=11pt]{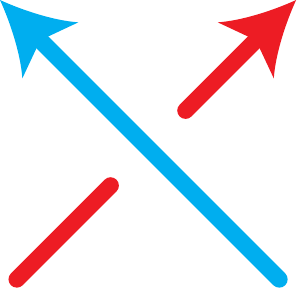}}}\newcommand\bs{\boldsymbol}
\newcommand\lk{\text{lk}}

\theoremstyle{plain}
\newtheorem{theorem}{Theorem}[section]
\newtheorem{lemma}[theorem]{Lemma}
\newtheorem{subl}[theorem]{Sublemma}
\newtheorem{obs}[theorem]{Observation}
\newtheorem{prop}[theorem]{Proposition}
\newtheorem{cor}[theorem]{Corollary}
\newtheorem{conjecture}[theorem]{Conjecture}
\newtheorem{corr}[theorem]{Correspondence}
\newtheorem{fact}[theorem]{Fact}
\newtheorem*{cblemma}{Crossing band lemma}
\newtheorem*{C:bk36}{Corollary 3.6 of \cite{bk20}}
\newtheorem*{T:bk1}{Proposition 3.8 of \cite{bk20}}
\newtheorem*{C:bk20}{Corollary 4.9 of \cite{bk20}}
\newtheorem*{T:411}{Theorem \ref{T:411}}
\newtheorem*{T:526}{Theorem \ref{T:526}}
\newtheorem*{T:kup}{Theorem 1 of \cite{kup03}}
\newtheorem*{T:mainthick}{Theorem \ref{T:mainthick}}
\newtheorem*{T:stable}{Theorem \ref{T:stable}}
\newtheorem*{T:virtual}{Theorem \ref{T:virtual}}
\newtheorem*{T:tait1}{Theorem \ref{T:tait1}}
\newtheorem*{T:tait2}{Theorem \ref{T:tait2}}
\newtheorem*{T:tait12v}{Theorem \ref{T:tait12v}}
\newtheorem*{T:ObviouslyPrimeEtc}{Theorem \ref{T:ObviouslyPrimeEtc}}
\newtheorem*{T:VSplitEtc}{Theorem 4.15 of \cite{primes}}
\newtheorem*{T:ess}{Theorem 1.1 of \cite{endess}}
\newtheorem*{C:NonUCS}{Corollary \ref{C:NonUCS}}

\theoremstyle{definition}
\newtheorem{convention}[theorem]{Convention}

\newtheorem{procedure}[theorem]{Procedure}
\newtheorem{notation}[theorem]{Notation}
\newtheorem{definition}[theorem]{Definition}

\newtheorem{example}[theorem]{Example}

\newtheorem{move}{Move}

\theoremstyle{remark}
\newtheorem{rem}[theorem]{Remark}

\numberwithin{equation}{section}

\begin{document}

\title{The virtual flyping theorem}

\author{Thomas Kindred}

\address{Department of Mathematics \& Statistics, Wake Forest University \\
Winston-Salem North Carolina, 27109} 

\email{thomas.kindred@wfu.edu}
\urladdr{www.thomaskindred.com}



\maketitle

\begin{abstract}
We extend the flyping theorem to alternating links in thickened surfaces and alternating virtual links.  The proof of the former result uses work of Boden--Karimi to adapt the author's geometric proof of Tait's 1898 flyping conjecture (first proved in 1993 by Menasco--Thistlethwaite), while the proof of the latter involves 
a diagrammatic correspondence recently introduced by the author in a related paper. 
In the process, we also extend a classical result of Gordon--Litherland, establishing an isomorphism between their pairing on a spanning surface and the intersection form on a 4-manifold constructed as a double-branched cover using that surface.
\end{abstract}


\section{Introduction}\label{S:intro}

P.G. Tait asserted in 1898 that all reduced alternating diagrams of a given prime nonsplit link in $S^3$ minimize crossings, have equal writhe, and are related by {\it flype} moves  (see Figure \ref{Fi:flype}) \cite{tait}. The first proofs came almost a century later, and all involved the Jones polynomial \cite{kauff,mur,mur87ii,this,menthis91,menthis93}.  In 2017, Greene gave the first {\it purely geometric} proof of part of the classical Tait conjectures \cite{greene}, and in 2020, the author gave the first purely geometric proof of Tait's flyping conjecture \cite{flyping}.

Recently, Boden, Chrisman, Karimi, and Sikora extended much of this to alternating links in thickened surfaces.  First, using generalizations of the Kauffman bracket, Boden--Karimi--Sikora proved that Tait's first two conjectures hold for alternating links in thickened surfaces \cite{bk18,bks19}.%
\footnote{Boden--Karimi proved Tait's first two conjectures for alternating links in thickened surfaces, with a few extra conditions \cite{bk18}, and with Sikora they extended those results to adequate links and removed the extra conditions \cite{bks19}.} 
Second, Boden--Chrisman--Karimi extended the Gordon--Litherland pairing to spanning surfaces in thickened surfaces \cite{bck21}. 
Third, Boden--Karimi applied this pairing to extend Greene's characterization of classical alternating links to links $L$ in thickened surfaces $\Sigma\times I$, %
proving that $L$ bounds connected definite surfaces of opposite signs if and only if $L$ is alternating and $(\Sigma\times I, L)$ is nonstabilized \cite{bk20}.%
\footnote{See \textsection\ref{S:Prime} for definitions of {\it stabilized}, {\it prime}, {\it locally {prime}}, {\it cellular}, {\it end-essential}, {\it definite}, and {\it removably nugatory}.}

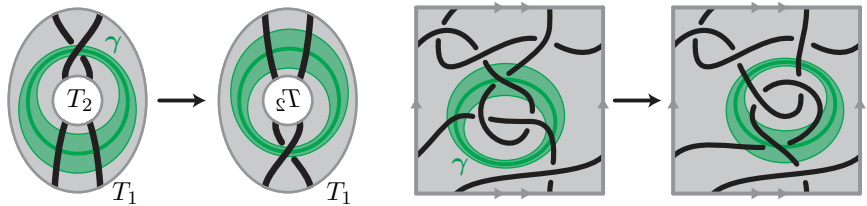
\begin{figure}
\begin{center}
\labellist
\hair4pt
\pinlabel {\scalebox{.9}{$T_2$}} [c] at 57 77
\pinlabel {\scalebox{.9}{$T_1$}} [c]  at 92 5
\pinlabel {\scalebox{.9}{\scalebox{-1}[1]{$T_2$}}} [c]  at 217 77
\pinlabel {\scalebox{.9}{$T_1$}} [c]  at 257 5
\pinlabel {\scalebox{.9}{$\FG{\boldsymbol{\gamma}}$}} [c]  at 83 120
\pinlabel {\scalebox{.9}{$\FG{\boldsymbol{\gamma}}$}} [c]  at 350 25
\endlabellist
\includegraphics[width=\textwidth]{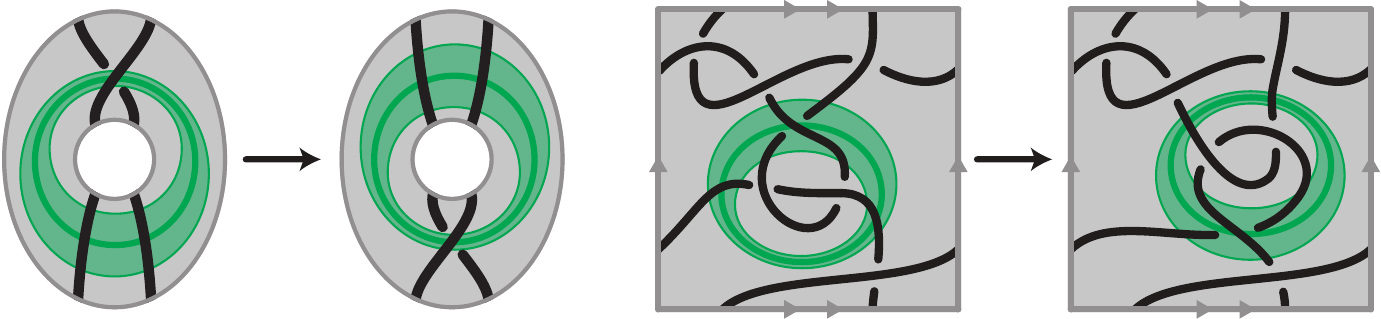}
\caption{A {\it flype} along an annulus $\FG{A=\nu\gamma}\subset \Sigma$.}
\label{Fi:flype}
\end{center}
\end{figure}

The first main result of this paper combines and adapts several of these recent developments to prove that the flyping theorem extends to alternating links in (nonstabilized) thickened surfaces. 

\begin{T:mainthick}
Let $D\subset \Sigma$ be a locally {prime}, %
 cellular alternating %
 diagram of a link $L$ in a thickened surface $\Sigma\times I$.%
   Then any other such diagram of $L$ is related to $D$ by flypes on $\Sigma$.
\end{T:mainthick}

The approach is parallel to that in \cite{flyping}, and indeed most of the arguments 
translate directly.  For some, which we mark with the symbol \gem, the statements and proofs hold without further comment. Appendix A lists pertinent cross-referencing information for these and other results marked with the symbol \ari. 
The upshot is a geometric proof of Theorem \ref{T:mainthick} and other generalized Tait conjectures:

\begin{T:tait1}[Part of Tait's extended first conjecture \cite{bk18,bks19}]
If $D,D'\subset\Sigma$ are alternating diagrams of a link $L\subset\Sigma\times I$, neither containing removable nugatory crossings, then $D$ and $D'$ have the same number of crossings.
\end{T:tait1}

\begin{T:tait2}[Tait's extended second conjecture \cite{bk18,bks19}]
All locally {prime}, cellular alternating diagrams of a given link $L\subset \Sigma\times I$ have the same writhe.
\end{T:tait2}

Section \ref{S:Virtual} extends Theorems \ref{T:tait1}, \ref{T:mainthick}, and \ref{T:tait2} to virtual links:

\begin{T:virtual}
Any two locally {prime}, alternating virtual diagrams%
\footnote{A virtual link diagram is {\it alternating} if its {\it classical crossings} alternate between over and under.}
of a given virtual link $K$ are related by non-classical R-moves and (classical) flypes.%
\footnote{A (classical) {\bf flype} on a virtual link diagram appears as in Figure \ref{Fi:flype}, where $T_1$ contains no virtual crossings.}
\end{T:virtual}

To prove this, we establish a new {\it diagrammatic} analog to the correspondence established by Kauffman, Kamada--Kamada, and Carter--Kamada--Saito between virtual links, equivalence classes of abstract links, and stable equivalence classes of links in thickened surfaces \cite{kauff98,kaka,cks02}; also see \cite{kup03}.  Given a virtual link diagram $V$, let $[V]$ denote its equivalence class under virtual (non-classical) R-moves.  We show that such classes $[V]$ correspond bijectively to abstract link diagrams and thus to cellularly embedded link diagrams on closed surfaces. 

As corollaries, we extend Theorems \ref{T:tait1} and \ref{T:tait2} to virtual diagrams, and we observe that connect sum is not a well-defined operation on virtual knots or links:

\begin{T:tait12v}
All locally {prime}, alternating diagrams of a given virtual link have the same crossing number and writhe.
\end{T:tait12v}

\begin{C:NonUCS}
Given any two non-classical, locally {prime}, alternating virtual links $V_1$ and $V_2$, there are infinitely many distinct virtual links that decompose as a connect sum of $V_1$ and $V_2$.
\end{C:NonUCS}

Before all this, in \textsection\ref{S:Back}, we introduce the required background regarding links in thickened surfaces. Some of this reviews the existing literature, and some of it is new.  

\section{Links and spanning surfaces in thickened surfaces}\label{S:Back}

\begin{convention}\label{Conv:Sigma}
Throughout, $\Sigma$ is a connected, closed, orientable surface 
with genus $g(\Sigma)>0$.%
%
\footnote{\cite{primes,endess} also allow $\Sigma$ to be disconnected with components of any genus.} 
We denote the intervals $[-1,1]$ and $[0,1]$ by $I$ and $I_+$, respectively. 
In $\Sigma\times I$, we identify $\Sigma$ with $\Sigma\times\{0\}$ and denote $\Sigma\times\{\pm1\}=\Sigma_\pm$.  
For a pair $(\Sigma,L)$ or $(\Sigma\times I, L)$, $L$ is a link in $\Sigma\times I$, and for a pair $(\Sigma,D)$, $D$ is a link diagram on $\Sigma$.
\end{convention}


\subsection{Alternating links in thickened surfaces}\label{S:Prime}
 
A pair $(\Sigma,L)$ is {\bf stabilized} if, for some circle%
\footnote{We use ``circle" as shorthand for ``simple closed curve."} 
 $\gamma\subset \Sigma$, $L$ can be isotoped so that it intersects each component of $(\Sigma\times I)\setminus(\gamma\times I)$ but not the annulus $\gamma\times I$; one can then {\it destabilize} the pair $(\Sigma,L)$ by cutting $\Sigma\times I$ along $\gamma\times I$ and attaching two 3-dimensional 2-handles in the natural way (this may disconnect $\Sigma$); the reverse operation is called {\it stabilization}. Equivalently, $(\Sigma,L)$ is {\it nonstabilized} if every diagram $D$ of $L$ on $\Sigma$ is {\bf cellularly embedded}, meaning that $D$ cuts $\Sigma$ into disks. 
 
 A pair $(\Sigma,L)$ is {\bf split} if $L$ has a disconnected diagram on $\Sigma$.  Note that if $(\Sigma,L)$ is split then it is also stabilized (as we assume that $\Sigma$ is connected). The converse is false. In fact, the number of split components is an invariant of stable equivalence classes.

Kuperberg's Theorem states that the stable equivalence class of $(\Sigma,L)$ contains a unique nonstabilized representative; this implies that when $(\Sigma,L)$ is nonsplit, $(\Sigma,L)$ is nonstabilized if and only if $\Sigma$ has {\it minimal genus} in this stable equivalence class.  


\begin{theorem}[Theorem 1 of \cite{kup03}]
If $(\Sigma,L)$ and $(\Sigma'\times I,L')$ are stably equivalent and nonstabilized, then there is a pairwise homeomorphism $(\Sigma\times I,L)\to(\Sigma'\times I,L')$.  
\end{theorem}

If $L$ is nonsplit %
and $g(\Sigma)>0$%
, then $(\Sigma\times I)\setminus L$ is irreducible, as $\Sigma\times I$ is always irreducible, since its universal cover is $\R^2\times \R$.%
\footnote{For more detail, see Proposition 12 of \cite{bk20}; the proof cites \cite{csw14}.}
The converse of this, too, is false,%
\footnote{If $(\Sigma_i\times I,L_i)$ is nonsplit (implying that $\Sigma_i\times I\setminus L_i$ is irreducible) for $i=1,2$, then choose disks $X_i\subset \Sigma_i$ with $(X_i\times I)\cap L_i=\varnothing$ and construct the connect sum $\Sigma=(\Sigma_1\setminus\text{int}(X_1))\cup(\Sigma_2\setminus\text{int}(X_2))=\Sigma_1\#\Sigma_2$. Let $L=L_1\sqcup L_2\subset\Sigma\times I$.  Then $(\Sigma,L)$ is split. Yet, $(\Sigma\times I)\setminus L$ is irreducible by Observation \ref{O:IrreducibleSplitUnion}.}
due to the next observation, which follows from a standard innermost circle argument:

\begin{obs}\label{O:IrreducibleSplitUnion}
If $(\Sigma_i\times I)\setminus L_i$ is irreducible for $i=1,2$ and $\Sigma=\Sigma_1\#_\gamma\Sigma_2$ with $L=L_1\sqcup L_2\subset \Sigma\times I$, where the annulus $A=\gamma\times I$ separates $L_1$ from $L_2$ in $\Sigma\times I$, then $(\Sigma\times I)\setminus L$ is irreducible.
\end{obs}

We call $(\Sigma,D)$ {\bf cellular} if $D$ cuts $\Sigma$ into disks. Note:

\begin{fact}[\cite{oz06,bk20}; Proposition 5.1 of \cite{primes}]\label{F:CB}
Suppose $D\subset \Sigma$ is a cellular alternating diagram of a link $L\subset \Sigma\times I$. Then $(\Sigma,D)$ is checkerboard colorable, and $L$ is nullhomologous over $\Z/2$.
\end{fact}

We will use this result of Boden--Karimi and the generalization that follows:

\begin{fact}[Corollary 3.6 of \cite{bk20}]\label{F:bk36}
If 
$(\Sigma,L)$ has a cellular alternating diagram, then $(\Sigma,L)$ is nonsplit and nonstabilized.
\end{fact}

\begin{cor}[Corollary 2.4 of \cite{primes}]\label{C:bk36++}
Suppose $(\Sigma,L)$ has an alternating diagram $D\subset\Sigma$. Then $(\Sigma,L)$ is nonsplit if and only if $D$ is connected, and $(\Sigma,L)$ is nonstabilized if and only if $D$ is cellular.
\end{cor}

Following \cite{primes}, we call a checkerboard colorable pair $(\Sigma,D)$ {\bf pairwise prime} if any pairwise connect sum decomposition $(\Sigma,D)=(\Sigma_1,D_1)\#(\Sigma_2,D_2)$ has $(\Sigma_i,D_i)=(S^2,\bigcirc)$ for either $i=1,2$.  
Likewise, given $(\Sigma,L)$ where $L$ is nullhomologous over $\Z/2$, we call $(\Sigma,L)$ {\it pairwise prime} if every annular connect sum decomposition $(\Sigma,L)=(\Sigma_1,L_1)\#(\Sigma_2,L_2)$ 
 is trivial: $(\Sigma_i,L_i)=(S^2,\bigcirc)$ for either $i=1,2$.\footnote{Annular connect sum $(\Sigma_1,L_1)\#(\Sigma_2,L_2)=(\Sigma,L)$ is a connect sum of surfaces, $\Sigma_1\#\Sigma_2=\Sigma$, thickened up, which restricts to a connect sum of 1-manifolds $L_1\#L_2=L$. See \cite{kauff98,mat,primes}.}
 \footnote{The definitions of pairwise primeness are more complicated without the assumptions related to checkerboard colorability; see \cite{primes}.}
Thus, such $(\Sigma,L)$ is pairwise prime if and only if, whenever $\gamma\subset\Sigma$ is a separating curve and $L$ is isotoped to intersect the annulus $\gamma\times I$ in two points, $\gamma$ bounds a disk $X\subset\Sigma$ such that $L$ intersects $X\times I$ in a single unknotted arc. 
 
 Howie-Purcell call $(\Sigma,D)$ {\it weakly prime} if, for every pairwise connect sum decomposition $(\Sigma,D)=(\Sigma,D_1)\#(S^2,D_2)$,  
either 
$D_2=\bigcirc$ is the trivial diagram of the unknot %
 or $(\Sigma,D_1)=(S^2,\bigcirc)$ \cite{hp20}; following \cite{primes}, we call such $D$ {\bf locally prime}. %
We call $(\Sigma,L)$ {\it locally prime} if, for every pairwise connect sum decomposition $(\Sigma,L)=(\Sigma,L_1)\#(S^2,L_2)$,
 either 
 $L_2=\bigcirc$ is the unknot %
 or $(\Sigma,L_1)=(S^2,\bigcirc)$ \cite{hp20}%
 .\footnote{A third notion of primeness for  $D$ on $\Sigma$ also appears in the literature: Ozawa calls $(\Sigma, D)$ {\it strongly prime} if every circle on $\Sigma$ (not necessarily separating) that intersects $D$ in two generic points also bounds a disk in $\Sigma$ which contains no crossings of $D$  \cite{oz06}.}

As in the classical case \cite{men84}, certain diagrammatic conditions constrain an alternating link $L$ as one might wish:

\begin{theorem}[\cite{oz06,bk20,adamsetal,primes}]\label{T:ObviouslyPrimeEtc}
If $D\subset \Sigma$ is a cellular alternating diagram of a link $L\subset \Sigma\times I$, then (i) $L$ is nonsplit, so in particular, $(\Sigma\times I)\setminus L$ is irreducible %
if $g(\Sigma)>0$; (ii) if $(\Sigma,D)$ is locally {prime}, then $(\Sigma,L)$ is locally {prime}; and (iii) if $(\Sigma,D)$ is pairwise prime, then $(\Sigma,L)$ is pairwise prime.%
%
\end{theorem}

Part (i) was proven by Ozawa in \cite{oz06} and by Boden-Karimi in \cite{bk20}. 
Part (ii) was proven by Adams et al in \cite{adamsetal} and generalized by Howie-Purcell in \cite{hp20}. Part (iii) is one of the main results of \cite{primes}, which also gives new proofs of (i)-(ii).

\subsubsection{End-essential spanning surfaces}
Part (i) of Theorem \ref{T:ObviouslyPrimeEtc} implies that $L$ has {\it spanning surfaces}: embedded, unoriented, compact surfaces $F\subset \Sigma\times I$ with $\partial F=L$; while we {\it do not} require 
$F$ to be connected, we do require that each component of $F$ has nonempty boundary. 
%
By deleting the interior of a regular neighborhood of $L$ from $F$ and $\Sigma\times I$, one may instead view $F$ as a properly embedded surface in the link exterior $(\Sigma\times I)\setminus\inter L$%
\footnote{Throughout, given a manifold $X$ and a submanifold $Y\subset X$,
$\nu Y$ denotes a {\it closed} regular neighborhood of $Y$ in $X$.}%
%
\footnote{We also assume that $\partial F$ is transverse on $\partial \nu L$ to each meridian, where a meridian is the preimage of a point in $L$ under the bundle map $\nu L\to L$.}
We take this view throughout, except in Definition \ref{D:essential}, Note \ref{N:S*}, and \textsection\ref{S:4Ball}.


If $(\Sigma,D)$ is a cellular alternating diagram of $(\Sigma,L)$, then it is possible to orient each disk of $\Sigma\setminus D$ so that, under the resulting boundary orientation, over- and under-strands are oriented respectively toward and away from crossings.  Since $\Sigma$ is orientable, these orientations determine a checkerboard coloring of $\Sigma\cut D$,\footnote{For compact $X,Y\subset \Sigma\times I$, $X{\cut} Y$ denotes the metric closure of $X\setminus Y$; see Note 7 of \cite{flyping} for a precise definition.} i.e. a way of shading the disks of $\Sigma\cut D$ black and white so that regions of the same shade abut only at crossings.\footnote{Interestingly, cellular alternating link diagrams on nonorientable surfaces are never {\bf checkerboard colorable}.} One can use this checkerboard coloring to construct {\it checkerboard surfaces} $B$ and $W$ for $L$, where $B$ projects into the black regions, $W$ projects into the white, and $B$ and $W$ intersect in {\it vertical arcs} which project to the the crossings of $D$.
The main result of \cite{endess} is that these checkerboard surfaces satisfy several convenient properties: 

\begin{definition}\label{D:essential}
Let $F\subset \Sigma\times I$ be a spanning surface for $(\Sigma,L)$.  
Denote $M_F=(\Sigma\times I)\cut F$, and use the natural map $h_F:M_F\to\Sigma\times I$ to denote $h_F^{-1}(L)=\wt{L}$, $h_F^{-1}(\Sigma_\pm)=\wt{\Sigma_\pm}$, and $h_F^{-1}(F)=\wt{F}$, so that $h_F:\wt{L}\to L$ and $h_F:\wt{\Sigma_\pm}\to\Sigma_\pm$ are homeomorphisms and $h_F:\wt{F}\setminus \wt{L}\to\text{int}(F)$ is a 2:1 covering map.  
Then we say that $F$ is:
\begin{enumerate}[label=(\alph*)]
\item {\bf incompressible} if any circle
$\gamma\subset \wt{F}\setminus\wt{L}$ that bounds a disk in $M_F$  
also bounds a disk in $\wt{F}\setminus\wt{L}$.%
\footnote{$F$ is incompressible if and only if $F$ is $\pi_1$-injective, meaning that inclusion $\text{int}(F)\hookrightarrow (\Sigma\times I)\setminus L$ induces an injection of fundamental groups (for all possible choices of basepoint).}
\item {\bf end-incompressible} if 
any circle %
$\gamma\subset \wt{F}\setminus\wt{L}$ that is parallel in $M_F$ to $\wt{\Sigma_\pm}$ bounds a disk in $\wt{F}\setminus\wt{L}$.
\item {\bf $\bs{\partial}$-incompressible} if, for any circle
$\gamma\subset \wt{F}$ with $|\gamma\cap\wt{L}|=1$ that bounds a disk in $M_F$, 
$\gamma\cut\wt{L}$ is parallel in $\wt{F}\cut\wt{L}$ into $\wt{L}$.
\item {\bf essential} if $F$ satisfies (a) and (c).
\item {\bf end-essential} if $F$ satisfies (b) and (c).\footnote{Note that any end-essential surface is essential. Observe moreover that the converse is true when $\Sigma$ is a 2-sphere. 
}
\end{enumerate}
\end{definition}

A crossing $c$ of a diagram $D\subset \Sigma$ is {\bf removably nugatory} if there is a disk $X\subset \Sigma$ such that $\partial X\pitchfork D=\{c\}$; in that case, one can remove $c$ from $D$ via a flype and a Reidemeister-1 move.  %
No  locally {prime} cellular diagram has removable nugatory crossings. Also, any diagram $(\Sigma,D)$ with a removable nugatory crossing, has at least one $\partial$-compressible checkerboard surface.  Conversely:

\begin{theorem}[Theorem 1.1 of \cite{endess}]\label{T:EndEss}
If $D\subset\Sigma$ is a cellular alternating diagram without removable nugatory crossings, then both checkerboard surfaces from $D$ are {end-essential}.
\end{theorem}




\begin{prop}
\label{P:BdryParallel}
Suppose $F_\pm$ are definite surfaces of opposite signs spanning a link $L\subset \Sigma\times I$ and $F_+\cap F_-$ consists only of arcs, none of which are $\partial$-parallel in both $F_+$ and $F_-$. If $F_-$ (resp. $F_+$) is $\partial$-incompressible, then no arc of $F_+\cap F_-$ is $\partial$-parallel in $F_+$ (resp. $F_-$).\gem
\end{prop}

\begin{prop}
\label{P:Ess}
If an essential surface $F$ spanning $(\Sigma,L)$ contains an arc $\beta$ which is parallel in $(\Sigma\times I)\cut(F\cup\nu L)$ to an arc $\alpha\subset\partial\nu L\cut\partial F$, then $\alpha$ is parallel in $\partial \nu L$ to $\partial F$.\gem
\end{prop}

\begin{obs}
\label{O:ArcW}
Suppose $B,W$ are the checkerboard surfaces of a cellular alternating diagram $D\subset\Sigma$ of a link $L\subset \Sigma\times I$. Any properly embedded arc in $W$ that is disjoint from $B$ and separating in $W$ is either $\partial$-parallel in $W$ or isotopic in $W$ to a vertical arc of $B\cap W$. Likewise with $B$ and $W$ reversed.\gem
\end{obs}

\begin{rem}
\label{R:CBEss}
Observation \ref{O:ArcW} implies in particular that no vertical arc from a locally {prime}, cellular alternating diagram is $\partial$-parallel in either checkerboard surface.\gem
\end{rem}

\subsubsection{Flype-related diagrams}\label{S:WellDefined}

\begin{definition}
\label{D:Flype}
If $D\subset \Sigma$ is a link diagram and $\gamma\subset \Sigma$ is an inessential circle that intersects $D$ transversally in three points, exactly one of them a crossing point, $c$, then we call the circle $\gamma$ a {\bf flyping circle} for $D$.
Up to mirror symmetry, $D$ and $\gamma$ appear as shown far left in Figure \ref{Fi:flype} ($D$ intersects the disk component of $\Sigma\setminus\inter\gamma$ in a tangle $T_2$ and intersects the other component in a ``higher-genus tangle'' $T_1$), so
one can {\bf flype} $D$ along $\gamma$  as shown: this move fixes $T_1$, switches which pair of strands cross within $\nu\gamma$, and changes $T_2$ by reflecting the underlying projection and reversing all crossing information.\ari 
\end{definition}




\begin{obs}
\label{O:Flype}
If $D\to D'$ is a flype, then $D$ and $D'$ represent the same link $L$ and have the same number of crossings. If $D$ is oriented then $D$ and $D'$ have the same writhe.%
\footnote{The {\bf writhe} of $D$ is $w_D=\left| \PosCrScr\right|-\left|  \NegCrScr\right|$.}
If $D$ is cellular alternating (resp. locally {prime}), then so is $D'$.\ari
\end{obs}

\begin{rem}\label{R:Entire}
In the classical setting, the tangle $T_1$ in Figure \ref{Fi:flype} might contain no crossings, in which case the flype has the effect of changing $D$ to its mirror image and then reversing all crossings; one may think of this move as leaving $D$ unchanged and viewing it from the opposite side of $\Sigma$ (in \cite{flyping}, we call such a flype an {\it entire flype}). By contrast (by an euler characteristic argument), no cellular checkerboard colorable diagram on a surface of positive genus does. Thus, while, as in \cite{flyping}, we regard two diagrams $D,D'\subset \Sigma$ as {\it equivalent} iff they are related by planar isotopy and possibly an entire flype, the latter possibility will be vacuous.
\end{rem}

\begin{figure}
\begin{center}
\includegraphics[width=.5\textwidth]{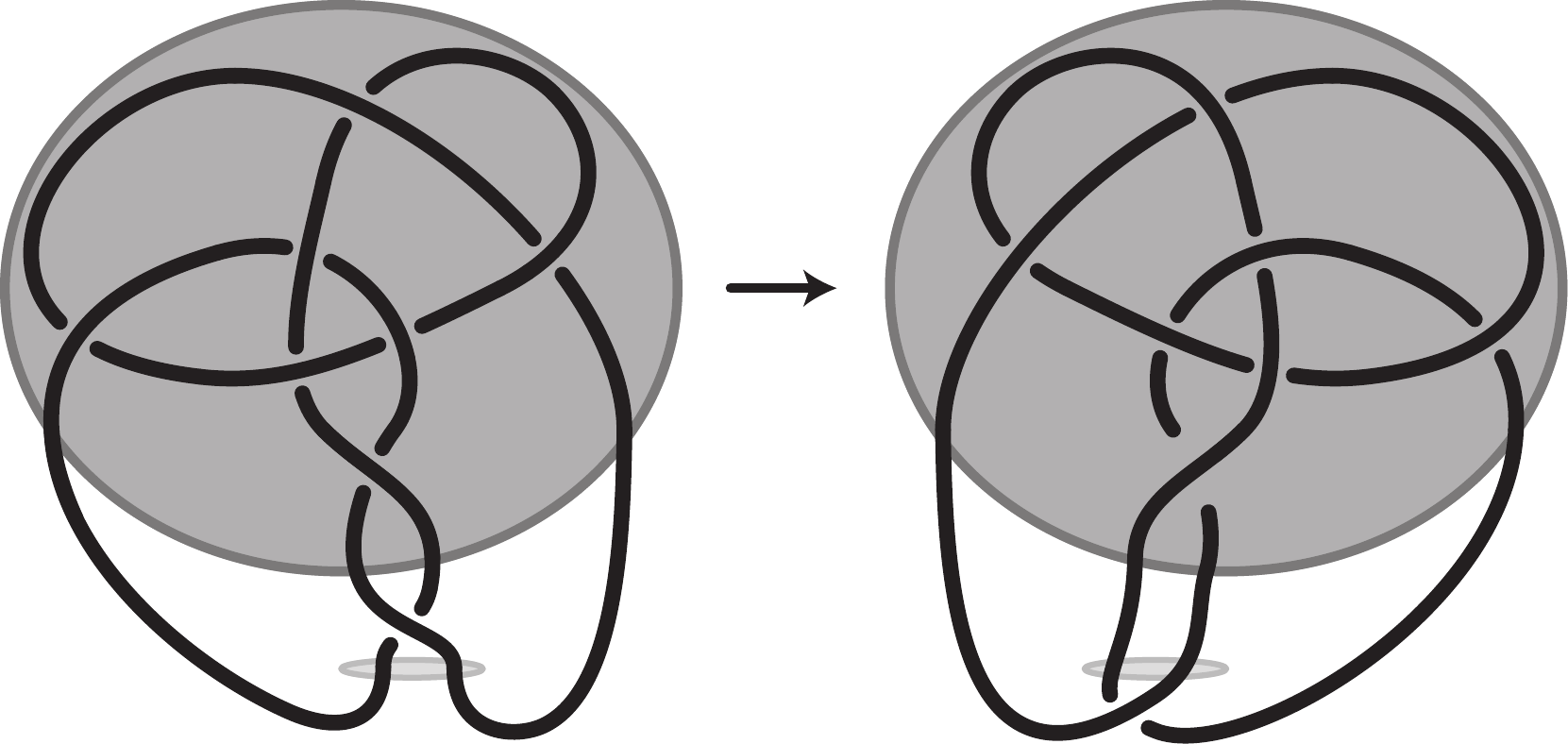}\hfill
\includegraphics[width=.45\textwidth]{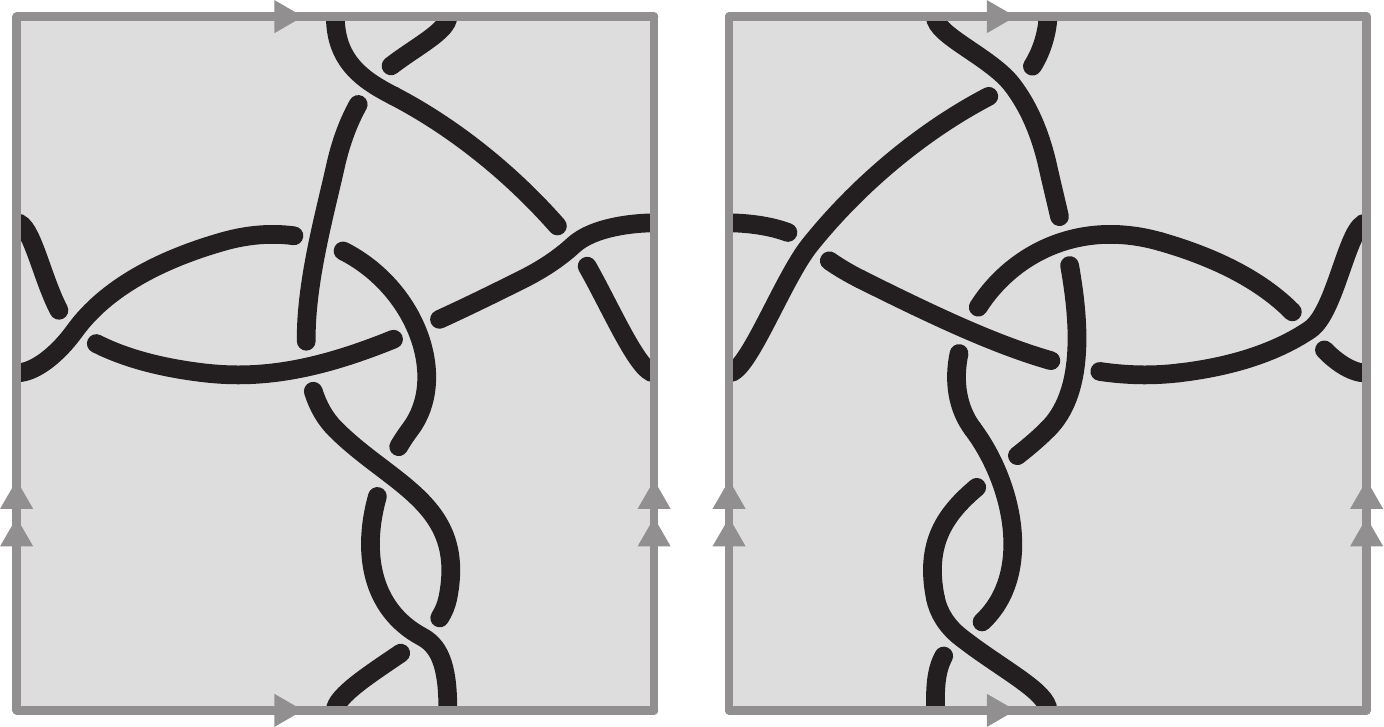}
\caption{Left: an {\it entire flype} of a 
 diagram of the knot 
$8_{17}$. Right: Corollary \ref{C:Entire} will imply that these links are non-isotopic; see Example \ref{Ex:Entire}.}
\label{Fi:Entire}
\end{center}
\end{figure}

\subsection{Definite surfaces}\label{S:definite}
  
\subsubsection{Linking numbers and slopes}\label{S:lk}

We adopt the notion of {\it generalized linking numbers} which was first defined for arbitrary 3-manifolds with nonempty boundary in \cite{ct07} and applied in the context of thickened surfaces in \cite{bck21,bk20}.  
The generalized linking number of disjoint multicurves%
\footnote{We call a disjoint union of embedded, {\it oriented} circles a {\bf multicurve}.}
$\red{\alpha},\Cyan{\beta}\subset\Sigma\times I$ is
\begin{equation}\label{E:lk}
\lk_\Sigma(\red{\alpha},\Cyan{\beta})=|\PosCrRB|-|\NegCrRB|.
\end{equation}
This linking pairing, taken  {\it relative to $\Sigma_+$}, is {\it asymmetric}: denoting intersection number on $\Sigma$ by $\cdot_\Sigma$ and projection $p_\Sigma:\Sigma\times I\to \Sigma$,
\[\lk_\Sigma(\red{\alpha},\Cyan{\beta})-\lk_\Sigma(\Cyan{\beta},\red{\alpha})=p_\Sigma(\red{\alpha})\cdot_\Sigma p_\Sigma(\Cyan{\beta}).\]
\color{black}

If $F$ spans a link $L=\bigsqcup_iL_i\subset \Sigma\times I$ 
and each $\widehat{L_i}$ is a co-oriented pushoff of $L_i$ in $F$, then 
we call $s(F)=\sum_{i}
\text{lk}(L_i,\widehat{L_i})$ the {\bf slope} of $F$.

\subsubsection{The Gordon--Litherland pairing}\label{S:GL}
\begin{figure}
\begin{center}
\includegraphics[width=.25\textwidth]{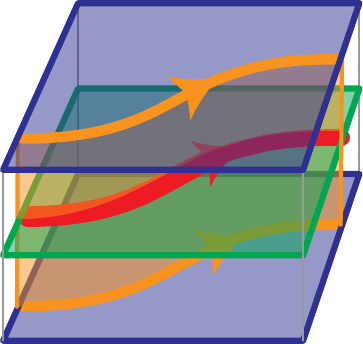}
\caption{A multicurve $\red{\gamma}\subset \FG{F}$ and $\Orange{\wt{\gamma}}\subset \Navy{\wt{F}}$: $[\wt{\gamma}]=\tau[\gamma]$.}\label{Fi:GL}
\end{center}
\end{figure}
Given a surface $F$ spanning a link $L\subset \Sigma\times I$, take $\nu F$ in the link exterior $(\Sigma\times I)\setminus \inter L$
with projection ${p}:{\nu}F\to F$,
such that $p^{-1}(\partial F)=\nu F\cap\partial\nu L$, and denote the {\it frontier} $\wt{F}=\partial\nu F\cut\partial\nu L$ and %
{\it transfer map} $\tau:H_1(F)\to H_1(\wt{F})$ (see Figure \ref{Fi:GL}). 
Following Boden--Chrisman--Karimi, the (generalized) {\it Gordon--Litherland pairing} (relative to $\Sigma_+$)
 is the symmetric bilinear mapping $\lla\cdot,\cdot\rra_F:H_1(F)\times H_1(F)\to\Z$
given by   \cite{gordlith,bck21}:
\[\lla a,b\rra_F=\frac{1}{2}\left(\lk_{\Sigma}(\tau a,b)+\lk_\Sigma(\tau b,a)\right).\]
\color{black}
Given a multicurve $\gamma\subset F$ representing $g\in H_1(F)$, we denote $\lla g,g\rra_F=\lb g\rb_F$ and call $\frac{1}{2}\lb g\rb_F$ the {\it framing} of $\gamma$ in $F$. %
Given a basis $\mathcal{B}=(a_1,\hdots,a_n)$ for $H_1(F)$, the {\it Goeritz matrix} $G=(x_{ij})\in\Z^{n\times n}$, $x_{ij}=\lla a_i,a_j\rra_F$, represents $\lla\cdot,\cdot\rra_F$ with respect to $\mathcal{B}$. %
Denoting the signature of $G$ by $\sigma(F)$, 
the quantity 
\begin{equation}\label{E:Sig}
\sigma_F(L)=\sigma(F)-\frac{1}{2}s(F), 
\end{equation}
depends only on the $S^*$ equivalence class of $F$; whenever $(\Sigma,L)$ is nonsplit with diagram $(\Sigma,D)$ there are exactly two such classes, each represented by a checkerboard surface of $D$ \cite{bck21}
.\footnote{\label{N:S*}$S^*$ equivalence is generated by attaching and deleting tubes and crosscaps \cite{gordlith} and thus respects relative homology classes.  The checkerboard surfaces $F$ and $F'$ of $D$ 
satisfy $[F]+[F']=[\Sigma]$ in $H_2(\Sigma\times I,L;\Z/2)$, so $[F]\neq [F']$; hence, $F$ and $F'$ are not $S^*$ equivalent.  For the converse, following the classical approach of Yasuhara \cite{yas}, put an arbitrary spanning surface in disk-band form, attach tubes to make it a checkerboard surface for some diagram, and then perform Reidemester moves (requiring more tubing and crosscapping moves).
}

\subsubsection{Definiteness characterizes alternating links.}\label{S:Def}

A spanning surface $F$ is
{\it positive-} (resp. {\it negative-}) {\it definite} if 
$\lla\alpha,\alpha\rra_F>0$ (resp. $\lla \alpha,\alpha\rra_F<0$) for all nonzero $\alpha\in H_1(F)$ \cite{greene}.%
\footnote{$F$ is positive-definite iff
$\sigma(F)=\beta_1(F)$ or equivalently iff each multicurve in $F$ either has positive framing in $F$ or bounds an orientable subsurface of $F$.
}%
\footnote{When $|\partial F|\leq 2 $, every primitive $g\in H_1(F)$ is represented by an oriented circle, but this is not true in general: e.g. take $F$ to be an oriented pair of pants and  $g$ the sum of two boundary components, one with the boundary orientation. 
}

Adapting work of Greene from the classical setting 
\cite{greene}, Boden--Karimi characterized nonstabilized alternating links in (and diagrams on) thickened surfaces in terms of definite surfaces:







\begin{fact}[Proposition 3.8 of \cite{bk20}]\label{F:PGreene}
A cellular checkerboard colorable link diagram $D\subset\Sigma$ is alternating if and only if its checkerboard surfaces are definite and of opposite signs.
\end{fact}

\begin{theorem}[Theorem 4.8 of \cite{bk20}]\label{T:bk48}
Suppose $(\Sigma,L)$ is nonstabilized.\footnote{Recall that this implies that $L\subset\Sigma\times I$ is a nonsplit link.} Then $L$ is alternating if and only if it has connected\footnote{\label{N:split}Spanning surfaces are assumed to be connected throughout \cite{bk20}.}
spanning surfaces of opposite signs. 
\end{theorem}
 
The proof in \cite{bk20} of Theorem \ref{T:bk48} shows moreover that if $L$ has connected spanning surfaces of opposite signs, then there is a closed surface $S$ in $\Sigma\times I$ on which $L$ has a cellular alternating diagram whose checkerboard surfaces are isotopic to the given surfaces; further, if $(L,\Sigma)$ is nonstabilized, then $S$ is isotopic to $\Sigma$.  Formally:
 
\begin{cor}\label{C:bk2}
If $(\Sigma,L)$ is nonstabilized and $B$ and $W$ are connected spanning surfaces of opposite signs spanning $L$, then $L$ has a cellular alternating diagram on $\Sigma$ whose checkerboard surfaces are isotopic to $B$ and $W$.%
\end{cor}

\begin{convention}\label{Conv:+-}
The checkerboard surfaces $B$ and $W$ of any cellular alternating diagram are labeled such that $B$ is positive-definite and $W$ is negative-definite. Likewise for checkerboard surfaces $B'$ and $W'$ (resp. $B_i$ and $W_i$) from such a diagram $D'$ (resp. $D_i$). 
%
\end{convention}

\begin{lemma}[c.f. \cite{bk20} Lemma 3.7]\label{L:bk37}
The checkerboard surfaces $B$ and $W$ of any cellular alternating diagram of a link $(\Sigma,L)$ satisfy%
\footnote{For an arbitrary diagram on $\Sigma$, $|\sigma_W(L)-\sigma_B(L)|\leq 2g(\Sigma).$}
\[\sigma_B(L)-\sigma_W(L)=2g(\Sigma).\]
\end{lemma}

Moreover, much of Boden--Karimi's proof of Theorem \ref{T:bk48} goes through even if the spanning surfaces of opposite signs for $L$ are disconnected or if $(\Sigma,L)$ is stabilized, or both. 
 In particular, if $L$ has spanning surfaces (not necessarily connected) of opposite signs, then there is a closed surface $S$ (not necessarily connected) in $ \Sigma\times I$ on which $L$ has a cellular alternating diagram $D$ whose checkerboard surfaces are isotopic to the given surfaces; further, each component of $S$ either is parallel to $\Sigma$ or is a 2-sphere. In particular:


\begin{fact}\label{F:SplitDef}
If $F_\pm$ are definite surfaces of opposite signs spanning a link $L\subset \Sigma\times I$, then for some (possibly empty) disjoint union of 2-spheres $\Sigma'\subset (\Sigma\times I)\setminus \Sigma$, $L$ has a cellular alternating diagram $D\subset \Sigma\cup \Sigma'$ whose checkerboard surfaces are isotopic to $F_\pm$. Thus:
\begin{enumerate}[label=(\Alph*)]
\item $F_+$ and $F_-$ have the same number of connected components, and this equals the number of split components of $L$.
\item $L$ has at most one non-local component.
\end{enumerate}
\end{fact}

\subsubsection{Intersections between definite surfaces}\label{S:sub}

Let $F$ and $F'$ be spanning surfaces for $(\Sigma,L)$ with $F\pitchfork F'$.
Orient $L$ arbitrarily, and orient $\partial F$ and $\partial F'$ so that each is homologous in $\nu L$ to $L$.  
Given an arc $\alpha$ of $F\cap F'$, take $\nu \partial\alpha$ in $\partial \nu L$. 
%
Following Howie \cite{howie}, we call $\alpha$ {\bf standard} if $i(\partial F,\partial F')_{\nu\partial\alpha}=\pm2$ and {\it non-standard} if $i(\partial F,\partial F')_{\nu\partial\alpha}=0$.
\begin{equation}\label{E:SlopeArcs}
s(F)-s(F')=i(\partial F,\partial F')_{\partial\nu L}=\sum_{\text{arcs }\alpha\text{ of }F\cap F'}i(\partial F,\partial F')_{\nu\partial\alpha}\\
\end{equation}


\begin{figure}
\begin{center}
\includegraphics[width=\textwidth]{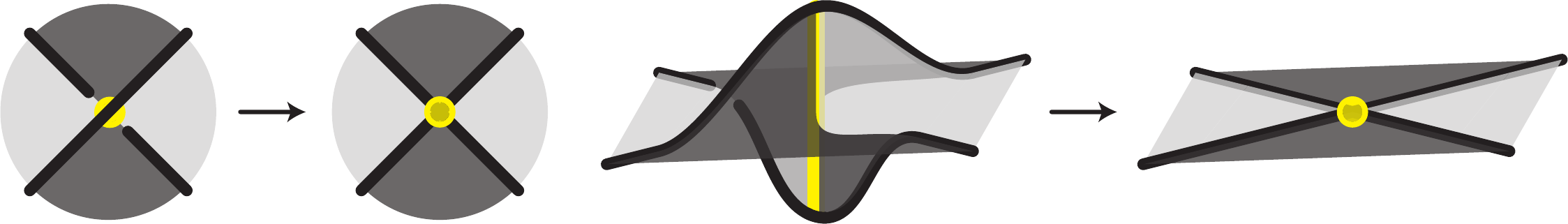}
\caption{Collapsing $S\cup T$ along a standard arc}
\label{Fi:ArcCollapse}
\end{center}
\end{figure}

\begin{procedure}
\label{Proc:ArcCollapse}
Let $(\Sigma,L)$ be non-stabilized with connected spanning surfaces $S,T$ such that $S\cap T$ consists entirely of standard arcs and $|S\cap T|=\beta_1(S)+\beta_1(T)+2g(\Sigma)$. Then extending $S,T$ through $\nu L$ so that $\partial S=L=\partial T$ and collapsing $S\cup T$ along each arc of $\text{int}(S)\cap \text{int}(T)$ gives a closed surface $Q$ isotopic to $\Sigma$\footnote{Connectedness and $|S\cap T|=\beta_1(S)+\beta_1(T)$ imply that $g(Q)=g(\Sigma)$. This and the assumption that $(\Sigma,L)$ is non-stabilized imply that $Q$ is isotopic to $\Sigma$.} on which $L$ collapses to a connected 4-valent graph; recovering crossing information gives a connected link diagram $D\subset Q$ for which $S$ and $T$ {\it are} checkerboard surfaces. The initial configuration of $S$ and $T$, up to isotopy of $S\cup T$ in $(\Sigma\times I)\setminus\inter L$, uniquely determines $D$ up to isotopy. See Figure \ref{Fi:ArcCollapse}.\ari 
\end{procedure}


\begin{prop}\label{P:GLS2}
If $(\Sigma,L)$ is local and has positive- and negative-definite connected spanning surfaces $F_+$ and $F_-$, then 
\[s(F_+)-s(F_-)=2\left(\beta_1(F_+)+\beta_1(F_-)\right).\]
\end{prop}
 
 \begin{proof}
 Because $L$ is local, the surfaces $F_+$ and $F_-$ are $S^*$-equivalent, so $\sigma_{F_+}(L)=\sigma_{F_-}(L)$, and the result follows from (\ref{E:Sig}).
 \end{proof}

\begin{prop}[c.f. Propositions 2.12 and 2.22 of \cite{flyping}]\label{P:DetermineD}
If $(\Sigma,L)$ is non-stabilized and has positive- and negative-definite connected spanning surfaces $F_+$ and $F_-$, then 
\[s(F_+)-s(F_-)=2\beta_1(F_+)+2\beta_1(F_-)+4g(\Sigma).\]
Further, if $F_+\cap F_-$ is comprised of arcs $\alpha$ with $i(\partial F_+,\partial F_-)_{\nu\partial\alpha}=+2$:
\begin{enumerate}[label=(\Alph*)]
\item  $|F_+\cap F_-|=\beta_1(F_+)+\beta_1(F_-)+2g(\Sigma)$, 
\item $F_\pm$ yield an alternating diagram $D$ via Procedure \ref{Proc:ArcCollapse}, and
\item if $F_+$ and $F_-$ are $\partial$-incompressible, then $D$ has no removable nugatory crossings.
\end{enumerate}
\end{prop}

\begin{proof}
Isotope $F_\pm$ so that each component $\alpha$ of $F_+\cap F_-$ is an arc with $i(\partial F_+,\partial F_-)_{\nu\partial\alpha}=+2$. Now
\[|F_+\cap F_-|=\frac{1}{2}|\partial F_+\cap\partial F_-|=\frac{1}{2}\left(s(F_+)-s(F_-)\right),\]
which equals $\beta_1(F_+)+\beta_1(F_-)+2g(\Sigma)$ by (\ref{E:Sig}) and Lemma \ref{L:bk37}
. Therefore,  the pair $F_\pm$ determines a connected diagram $D$ of $L$ via Procedure \ref{Proc:ArcCollapse}. The checkerboard surfaces of $D$ are $F_\pm$, so $D$ is alternating by Fact \ref{F:PGreene}. Part (C) follows easily.
\end{proof}

\begin{fact}[c.f. Fact 2.23 of \cite{flyping}, Lemma 3.4 of \cite{greene}]\label{F:GreeneCircle}
If $F_+\pitchfork F_-$ are definite surfaces of opposite signs spanning a link $L\subset \Sigma\times I$, then any circle $\gamma\subset F_+\cap F_-$  bounds disks in both $F_+$ and $F_-$.
\end{fact}


\begin{procedure}
\label{Proc:Kill1}
Suppose $F_+\pitchfork F_-$ are definite surfaces of opposite signs spanning a link $L\subset \Sigma\times I$.
Fixing $F_-$, isotope $F_+$ via the following hierarchy of moves:\footnote{That is, perform (1) whenever possible,  perform (2) whenever possible unless (1) is possible, and perform (3) whenever possible unless (1) or (2) is possible.
}
\begin{figure}
\begin{center}
\includegraphics[width=.8\textwidth]{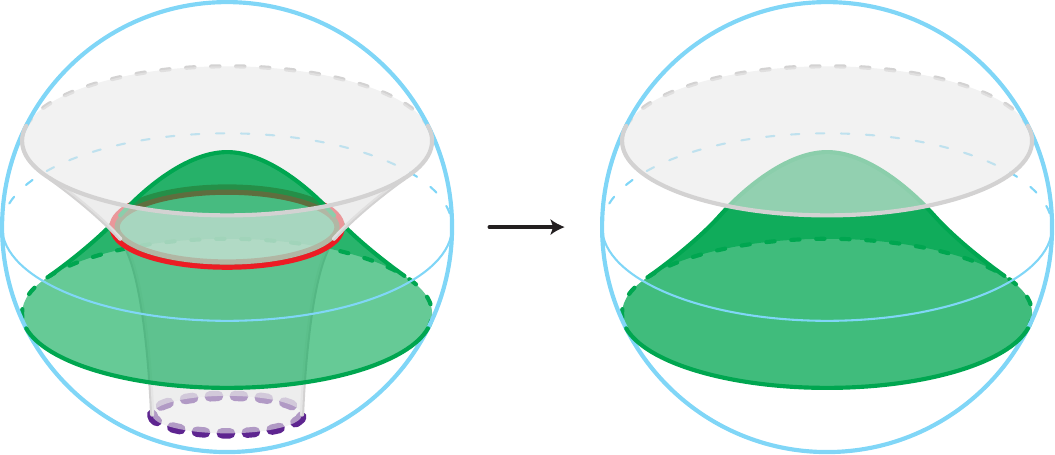}
\caption{Removing a $\red{\text{circle }\gamma}$ of intersection between positive- and negative-definite surfaces $\FG{F_+}$ and $\Gray{F_-}$.  The dashed purple circle bounds a disk in $F_+$.}
\label{Fi:RemoveCircle}
\end{center}
\end{figure}
\begin{enumerate}[label=(\arabic*)]
\item If $F_+\cap F_-$ contains circles, then (using Fact \ref{F:GreeneCircle}) choose an innermost one $\gamma$ in $F_+$; $\gamma$ bounds disks $X_\pm\subset F_\pm$.  Using the irreducibility of $(\Sigma\times I)\setminus L$, isotope $X_+$ past $X_-$ as shown in Figure \ref{Fi:RemoveCircle}. Meanwhile, fix $F_+$ away from $X_+$.
\item If any arc $\alpha$ of $F_+\cap F_-$ is parallel in $F_-{\cut} F_+$ to $\partial F_-$ \emph{and} in $F_+{\cut} F_-$ to $\partial F_+$, then remove $\alpha$ as shown in Figure \ref{Fi:FWMove}, top.
\begin{figure}
\begin{center}
\includegraphics[width=\textwidth]{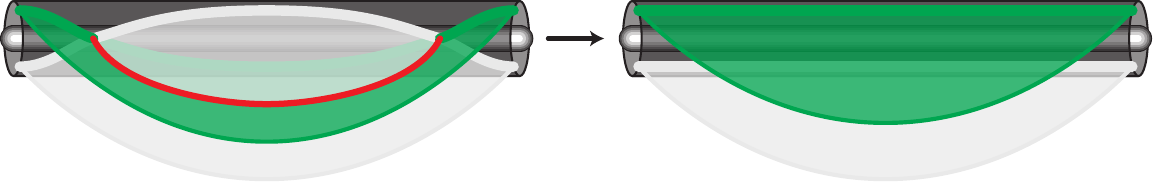}\\
\includegraphics[width=\textwidth]{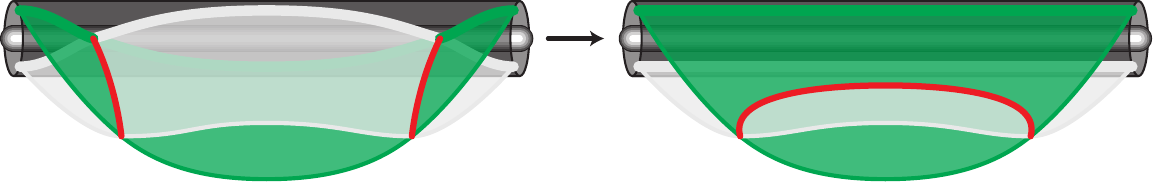}
\caption{Removing adjacent points of $\partial F_+\cap\partial F_-$ of opposite sign}
\label{Fi:FWMove}
\end{center}
\end{figure}
\item 
If arcs $\alpha_+\subset \partial F_+{\cut}\partial F_-$ and $\alpha_-\subset \partial F_-{\cut}\partial F_+$ are parallel in $\partial\nu L$, then push $\alpha_+$ past $\alpha_-$ as in Figure \ref{Fi:FWMove}, bottom.\ari
\end{enumerate}
\end{procedure}

We also recall:

\begin{fact}
\label{Fact:arccutdef}
If $\alpha$ is a system of disjoint properly embedded arcs in a definite surface $F$, then  $F\setminus\overset{\scriptscriptstyle{\circ}}{\nu}{\alpha}$ is definite.\ari
\end{fact}
\begin{figure}
\begin{center}
\labellist
\tiny\hair 4pt
\pinlabel {$\white{\boldsymbol{\alpha}}$} [c] at 113 127
\pinlabel {$\white{\boldsymbol{\alpha'}}$} [c] at 110 50
\pinlabel {$\white{\boldsymbol{\alpha''}}$} [c] at 412 65
\endlabellist
\includegraphics[width=\textwidth]{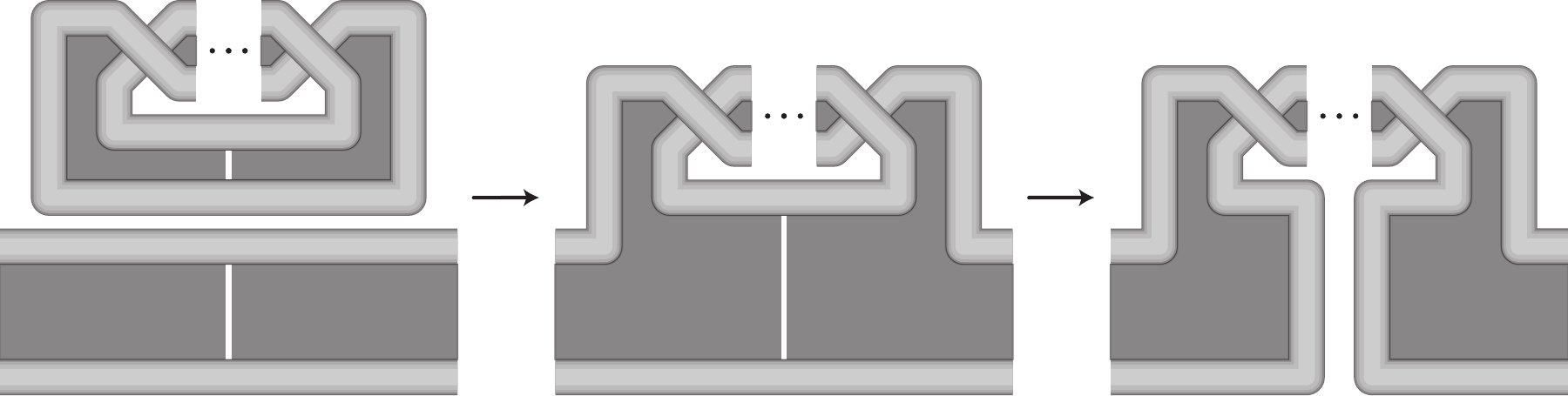}
\caption{Adding positive twists to a spanning surface}
\label{Fi:AddTwists}
\end{center}
\end{figure}
\begin{fact}
\label{F:AddTwist}
If $F'$ is obtained by adding positive twists to a positive-definite surface $F$ as in Figure \ref{Fi:AddTwists}, then $F'$ is positive-definite.\ari%
\footnote{Likewise for adding negative twists to a negative-definite surface.}%
\end{fact}

\begin{fact}
\label{F:Arc0}
If $F_\pm$ are definite surfaces of opposite signs spanning $(\Sigma,L)$ and $\alpha$ is a non-standard arc of $F_+\cap F_-$, then denoting $F_+'=F_+\setminus\inter\alpha$, $L'=\partial F_+'$, and $F_-'=F_-\setminus\inter\alpha$, the following are equivalent:
\begin{enumerate}[label=(\Roman*)]
\item $\alpha$ is separating on $F_+$;
\item $\alpha$ is separating on $F_-$;
\item $L'$ has one more split component than $L$.\ari
\end{enumerate}
\end{fact}


The next two facts differ notably from their classical analogs:

\begin{fact}[c.f. Proposition 6.6 of \cite{flyping}]\label{F:BEss}
Let $F$ be a positive-definite surface spanning a locally {prime} alternating link $L$, and let $K$ be the kernel of the map $H_1(F)\to H_1(\Sigma\times I)$ induced by inclusion $F\hookrightarrow \Sigma\times I$.  Then $F$ is end-essential if and only if every nonzero $a\in K$ satisfies $\lla a,a\rra_F\geq 2$.\footnote{Definiteness implies that $F$ is end-incompressible.}
\end{fact}

\begin{proof}
Take an end-essential negative-definite spanning surface $W$ for $L$ with $W\pitchfork F$, and let $D$ be an alternating diagram of $L$ associated to $F,W$ (via Procedure \ref{Proc:Kill1} and then \ref{Proc:ArcCollapse}).  If $D$ is locally {prime}, then both conditions are satisfied, the first by Theorem \ref{T:EndEss} and the second by an argument analogous to the proof of Lemma 4 of \cite{cromur}. 
Conversely, if $D$ admits a removable nugatory crossing $c$, then neither condition holds, because $W$ is end-essential.
\end{proof}

\begin{prop}[c.f. Proposition 6.7 of \cite{flyping}]\label{P:CutEss}
Let $F$ be a positive-definite surface spanning a locally {prime} alternating link $L$, and let $\alpha\subset F$ be a properly embedded arc such that $F'=F\setminus\inter\alpha$ spans a locally {prime} alternating link $L'$.  If $F$ is end-essential, then $F'$ is also end-essential.
\end{prop}

\begin{proof}
Letting  $K$ and $K'$ denote the kernels of the maps $H_1(F)\to H_1(\Sigma\times I)$ and $H_1(F')\to H_1(\Sigma\times I)$ induced by inclusion, Fact \ref{F:BEss} tells us that every nonzero $c\in K$ satisfies $\lla c,c\rra_F\geq 2$, and Fact \ref{Fact:arccutdef} implies that $F'$ is positive-definite. Therefore every nonzero $c\in K'$ satisfies $\lla c,c\rra_F\geq 2$, and so Fact \ref{F:BEss} implies that $F'$ is end-essential.
\end{proof}

\begin{prop}
\label{P:Kill1}
As a result of Procedure \ref{Proc:Kill1}, $F_+\cap F_-$ consists only of standard positive arcs.\gem%
\footnote{Note that Procedure \ref{Proc:Kill1} always terminates because each move decreases $|F_+\cap F_-|+|\partial F_+\cap\partial F_-|$.}
\end{prop}

\begin{prop}
\label{P:BdryParallel0}
If $F_\pm$ are definite surfaces of opposite signs spanning a link $L\subset\Sigma\times I$ and $\alpha$ is an arc of $F_+\cap F_-$ that is $\partial$-parallel in both $F_+$ and $F_-$, then $\alpha$ is non-standard.\gem
\end{prop}

\begin{lemma}[c.f. Lemma 2.30 of \cite{flyping}]\label{L:--Only}
Suppose $F_\pm$  are positive- and negative-definite surfaces spanning a non-stabilized link $L\subset \Sigma\times I$, and $\alpha$ is an arc of $F_+\pitchfork F_-$. Then:
\begin{enumerate}[label=(\Alph*)]
\item $i(\partial F_+,\partial F_-)_{\nu\partial \alpha}\neq-2$. 
\item If $\alpha$ is nonseparating on $F_-$, then $i(\partial F_+,\partial F_-)_{\nu\partial \alpha}=2$. 
\item In particular, if $L$ is locally {prime}, both $F_\pm$ are essential, and $\alpha$ is not $\partial$-parallel in both $F_\pm$, then $i(\partial F_+,\partial F_-)_{\nu\partial \alpha}=2$.
\end{enumerate}
\end{lemma}


\begin{proof}
The argument is largely the same as in \cite{flyping}. For (A) and (B), we just describe the differences: if $(\Sigma,L')$ is nonstabilized, then replacing $\beta_1(F_+)+\beta_1(F_-)$ with $\beta_1(F_+)+\beta_1(F_-)+2g(\Sigma)$ in (6.1) and (6.2) of \cite{flyping} contradicts Proposition \ref{P:DetermineD} (A); if $(\Sigma,L')$ is stabilized, then Fact \ref{F:SplitDef} (A) (and, for (B), the assumption that $\alpha$ is non-separating on $F_-$) implies that $L'$ is local, so Proposition \ref{P:GLS2} gives:
\begin{equation}
\begin{split}\label{E:No++B}
-2&=
\left(s(F_+)-s(F_-)\right)-\left(s(F_+')-s(F'_-)\right)\\
-2&=2 (\beta_1(F_+)+\beta_1(F_-)+2g(\Sigma))-2 (\beta_1(F_+')+\beta_1(F'_-))\\
-1&=g(\Sigma).
\end{split}
\end{equation}
We prove (C) 
by contradiction. Apply Procedure \ref{Proc:Kill1} $F_+=F_0\to F_1\to\cdots\to F_t$ until it terminates, and consider the last move (3) $F_s\to F_{s+1}$ in the sequence, which involves  two arcs $\alpha_1,\alpha_2$ of $F_s\cap F_-$ and one arc $\alpha$ of $F_{s+1}\cap F_-$; perturb $\alpha_1$ in $F_-$ so that it is disjoint from $F_s$. Parts (A) and (B) imply without loss of generality that $\alpha_1$ is non-standard, so $F_-\setminus\nu\alpha_1$ and $F_s\setminus\nu\alpha_1$ are definite surfaces of opposite sign spanning the same link $L'$.  Observe that, for {all} $i=s+1,\hdots, t$ (c.f. (6.3) of \cite{flyping}), and each arc $\alpha'$ of $F_-\cut F_i$ that separates $F_-$, either $\alpha'$ is $\partial$-parallel in $F_-$ or $\partial(F_-\setminus\nu\alpha')$ is split with no local components.  The latter ``possibility'' uses the assumption that $L$ is locally prime; it also contradicts Fact \ref{F:SplitDef} (B).  Therefore,  $\alpha_1$ is $\partial$-parallel in $F_-$, which contradicts the hierarchy of the moves in Procedure \ref{Proc:Kill1}.
\end{proof}

Using Lemma \ref{L:--Only}, the same reasoning as in \cite{flyping} leads to:

\begin{theorem}
\label{T:DBW}
Suppose $(\Sigma,D)$ and $(\Sigma,D')$ are locally prime, cellular alternating diagrams of $(\Sigma,L)$ with checkerboard surfaces $B,W$ and $B',W'$. Then $D$ and $D'$ are 
equivalent if and only if $B$ and $B'$ are isotopic in $(\Sigma\times I)\setminus\inter L$, as are $W$ and $W'$.\gem
\end{theorem}

\begin{cor}
\label{C:DBW}
There is a bijective correspondence between equivalence classes of locally prime, cellular alternating link diagrams on $\Sigma$ and pairs of isotopy classes of essential definite surfaces of opposite signs spanning the same locally prime, nonstabilized link in $\Sigma\times I$.\gem\footnote{Example 2.37 of \cite{flyping} shows that Theorem \ref{T:DBW} and Corollary \ref{C:DBW} become false if one removes ``locally prime" or ``cellular alternating."}
\end{cor}

\subsection{
Plumbing}\label{S:plumb}

A {\it plumbing cap} for a surface $F$ spanning $(\Sigma,L)$  is an embedded disk $V\subset (\Sigma\times I)\setminus\inter L$ with $V\cap(F\cup\partial\nu L)=\partial V$ where:
\begin{itemize}
\item 
$\partial V$ bounds a disk $\widehat{U}\subset F\cup \nu L$, 
\item  $\widehat{U}\cap F$ is a disk $U$ called the {\it shadow} of $V$, and
\item denoting the components of $(\Sigma\times I)\cut(\widehat{U}\cup V)$ by  $Y_1,Y_2$, neither subsurface $F_i=F\cap Y_i$ is a disk.
\end{itemize}
If the first two properties hold but the third fails, we call $V$ a {\it fake plumbing cap} for $F$.\footnote{The decomposition $F=F_1\cup F_2$ is a {\it de-plumbing} of $F$ along $U$ and $V$, denoted $F=F_1*F_2$. The reverse operation, in which one obtains $F$ by gluing $F_1$ and $F_2$ along $U$, is called {\it generalized plumbing} or {\it Murasugi sum}.}
\begin{figure}
\begin{center}
\includegraphics[width=.8\textwidth]{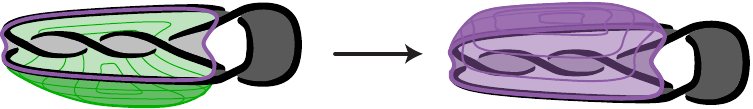}
\caption{Re-plumbing a spanning surface replaces a plumbing \FG{shadow } with its \violet{cap}\color{black}.}\label{Fi:replumb}
\end{center}
\end{figure}
If $V$ is a plumbing cap for $F$ with shadow $U$, then the operation $F\to (F\setminus U)\cup V$ is called {\bf re-plumbing}. See Figure \ref{Fi:replumb}.  The same operation along a fake plumbing cap, a ``fake re-plumbing,'' is an isotopy move.
Two spanning surfaces are {\it plumb-related} if they are related by re-plumbing and isotopy moves.
 

\subsubsection{The 4-dimensional perspective}\label{S:4Ball}
\begin{prop}[c.f. Proposition 2.36 of \cite{flyping}]\label{P:B4}
Given surfaces $F_1,F_2$ spanning $(\Sigma,L)$, let $F'_i$ be properly embedded surfaces in $\Sigma\times I\times I_+$ obtained by perturbing $\text{int}(F_i)$, while fixing $\partial F_1=L=\partial F_2$. 
If $F_1\setminus\inter L$ and $F_2\setminus\inter L$ are plumb-related, then:
\begin{enumerate}[label=(\Alph*)]
\item $F'_1$ and $F'_2$ are related by an ambient isotopy of $\Sigma\times I \times I_+$ which fixes $\Sigma\times I\supset L$;
\item there is an isomorphism $\phi:H_1(F_1)\to H_1(F_2)$ satisfying $\lla \alpha,\beta \rra_{F_1}=\lla \phi(\alpha),\phi(\beta)\rra_{F_2}$ for all $\alpha,\beta\in H_1(F_1)$;
\item if $F_1$ is definite, then $F_2$ is definite of the same sign;
\item in particular, if $F_1$ is a checkerboard surface from an alternating diagram of $L$ on $\Sigma$, then so is $F_2$;
\item $F_1$ and $F_2$ are $S^*$ equivalent, and thus $\sigma_{F_1}(L)=\sigma_{F_2}(L)$.
\end{enumerate}
\end{prop}

\begin{proof}
Part (A) is the same as in \cite{flyping}. For (B), construct the desired isomorphism $\phi:H_1(F_1)\to H_1(F_2)$ as follows.  Given $a\in H_1(F_1)$, take a multicurve $\alpha\subset F_i$ representing $a$, replace each arc of $\alpha\cap U$ with an arc in $V$ (with the same initial and terminal points), and denote the resulting multicurve by $\alpha'$; set $\phi(a)=[\alpha']$.  
This immediately gives (C) and (D), and (E) now follows from the observation that $[F_1]+[F_2]=0\in H_2(\Sigma\times I,L;\Z/2)$, since the union of any plumbing cap and its shadow is nullhomologous.
\end{proof} 

Next, we extend Theorem 3 of \cite{gordlith} to the context of thickened surfaces. Let $F$ be a spanning surface of a link $L\subset\Sigma \times I$. Isotope $F$ so that $F\subset (\Sigma\setminus\inter x)\times I$ for some point $x\in \Sigma$.\footnote{To see that this is always possible, consider isotoping $F$ into disk-band form.} 
Let $F'$ be a properly embedded surface in $(\Sigma\setminus\inter x)\times I\times I_+$ obtained by perturbing the interior of $F$ while fixing $\partial F$. One can construct the double-branched cover $M_{\wh{F}}$ of $(\Sigma\setminus\inter x)\times I\times I_+$ along $F'$ by cutting $\Sigma\times I\times I_+$ along the trace of this isotopy, taking two copies, and gluing. Yet, these two copies are homeomorphic to $\Sigma\times I\times I_+$, and the gluing region corresponds to a regular neighborhood $N$ of $F$ in $\Sigma\times I$.
Therefore, one may instead construct $M_{\wh{F}}$ as follows. Let $\iota:N\to N$ be involution given by reflection in the fiber, take two copies $\Sigma^4_1$ and $\Sigma^4_2$ of $(\Sigma\setminus\inter x)\times I\times I_+$, and define
\[M_{\wh{F}}=\left(\Sigma^4_1\cup\Sigma^4_2\right)/\left(y\in N\subset\partial\Sigma^4_1\sim\iota(y)\in N\subset \partial\Sigma^4_2\right).\]
Consider the Mayer-Vietoris sequence for $M_{\wh{F}}$:
\[0=H_2(\Sigma^4_1)\oplus H_2(\Sigma^4_2)\!\to\! H_2(M_{\wh{F}})\overset{\varphi}{\!\to\!} H_1(N)\overset{\psi}{\!\to\!} H_1(\Sigma^4_1)\oplus H_1(\Sigma^4_2)\!\to\! \cdots\]
If $g(\Sigma)=0$, as in \cite{gordlith}, then both $\Sigma^4_i$ are 4-balls, so $\varphi$ is an isomorphism; Gordon--Litherland then use the inverse map to compare the intersection form $\cdot$ on $M_{\wh{F}}$ with their pairing $\Gs_F$ on $F$.  After restricting appropriately, the same ideas work here:

\begin{theorem}[c.f. Theorem 3 of \cite{gordlith}]\label{T:gordlith}
With the setup above, let $i_*:H_1(F)\to H_1(N)$ be the isomorphism induced by inclusion, and denote $K=i_*^{-1}\left(\ker(\psi)\right)$.  Then there is an isomorphism $S:
\left(K,\Gs_F\right)\to \left(H_2(M_{\wh{F}}),\cdot\right)$.
\end{theorem}

\begin{proof}
Consider the following map $S:K \to H_2(M_{\wh{F}})$. Given $A\in K$, choose a multicurve $\alpha\subset F$ with $[\alpha]=A$.  
Then $\alpha$ bounds properly embedded oriented surfaces $s_i\subset \Sigma^4_i$ for $i=1,2$.  Define $S(A)=[s_1]-[s_2]\in H_2(M_{\wh{F}})$.

To see that this is the required isomorphism $\left(K,\Gs_F\right)\to \left(H_2(M_{\wh{F}}),\cdot\right)$, let $A,B\in K$, represented respectively by multicurves $\alpha,\beta\subset F$. Then 
$\alpha$ and $\wt{\beta}$ are disjoint multicurves in $N$ with $[\wt{\alpha}]=2A$, $[\wt{\beta}]=2B$, $\iota(\alpha)=\alpha$, and $\iota(\wt{\beta})=\wt{\beta}$.  Hence:
\begin{align*}\pushQED{\qed}
S(A)\cdot S(B)
&=\frac{1}{4}\left(S([\wt{\alpha}])\cdot S([\beta])+S([\wt{\beta}])\cdot S([\alpha])\right)\\
&=\frac{1}{4}\left(\lk_\Sigma\left(\wt{\alpha},\beta\right)+\lk_\Sigma\left(\iota\wt{\alpha},\iota\beta\right)+\lk_\Sigma(\wt{\beta},\alpha)+\lk_\Sigma(\iota\wt{\beta},\iota\alpha)\right)\\
&=\frac{1}{2}\left(\lk_\Sigma\left(\wt{\alpha},\beta\right)+\lk_\Sigma(\wt{\beta},\alpha)\right)\\
&=\Gs_F(A,B).\qedhere
\end{align*}
\end{proof}


\subsubsection{Flyping caps}
Let $D\subset \Sigma$ be a locally {prime}, cellular alternating link diagram with checkerboard surfaces $B,W$.  Say that a plumbing cap $V$ for $B$ is a {\bf flyping cap}  if $V$ appears as in Figure \ref{Fi:FlypeCap}, left-center. There is then a corresponding flype move as shown in Figures \ref{Fi:FlypeCap} and \ref{Fi:flypereplumb}.  Namely, denoting the shadow of $V$ by $U$, the flype move proceeds along an annular neighborhood of a circle $\gamma\subset \Sigma$ comprised of the arc $V\cap W$ together with an arc in $U\cup\nu L$.  
(The resulting link diagram might be equivalent to $D$.) 
More formally:

\begin{prop}[c.f. Proposition 2.37 of \cite{flyping}]\label{P:flypereplumb}
Let $V$ be an flyping cap for $B$, $D\to D'$ the flype move corresponding to $V$, $B'$ and $W'$ the checkerboard surfaces from $D'$, and $B''$ the surface obtained by re-plumbing $B$ along $V$. Then $B'$ and $B''$ are isotopic, as are $W'$ and $W$. Hence, $D'$ is equivalent to the diagram determined by $B'',W$ via Theorem \ref{T:DBW}.%
\footnote{An analogous statement holds for flyping caps for $W$.}
\end{prop}

\begin{figure}
\begin{center}%
\includegraphics[width=\textwidth]{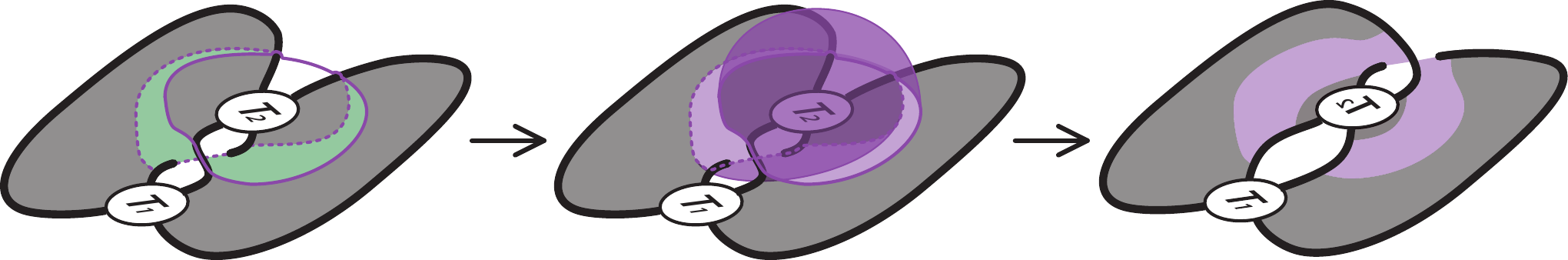}
\caption{A flype move corresponds to an isotopy of one checkerboard surface (here, $W$) and a re-plumbing of the other.}\label{Fi:flypereplumb}
\end{center}
\end{figure}

\begin{proof}
As in \cite{flyping}, Figure \ref{Fi:flypereplumb} demonstrates the isotopies.
\end{proof}

Conversely, if $\gamma$ is a flyping circle for $(\Sigma,D)$, then there is an flyping cap $V$ for $B$ (or $W$) with $V\cap W\subset \nu \gamma$ (resp. $V\cap B\subset \nu \gamma$). 

\begin{figure}
\labellist
\small\hair 4pt
\pinlabel \scalebox{.9}{Tangle 2} [c] at 744 247
\pinlabel \scalebox{.9}{Tangle 1} [c] at 744 -30
\pinlabel \scalebox{.9}{\reflectbox{Tangle 2}} [c] at 1360 247
\pinlabel \scalebox{.9}{Tangle 1} [c] at 1360 -30
\endlabellist
\begin{center}
\includegraphics[width=\textwidth]{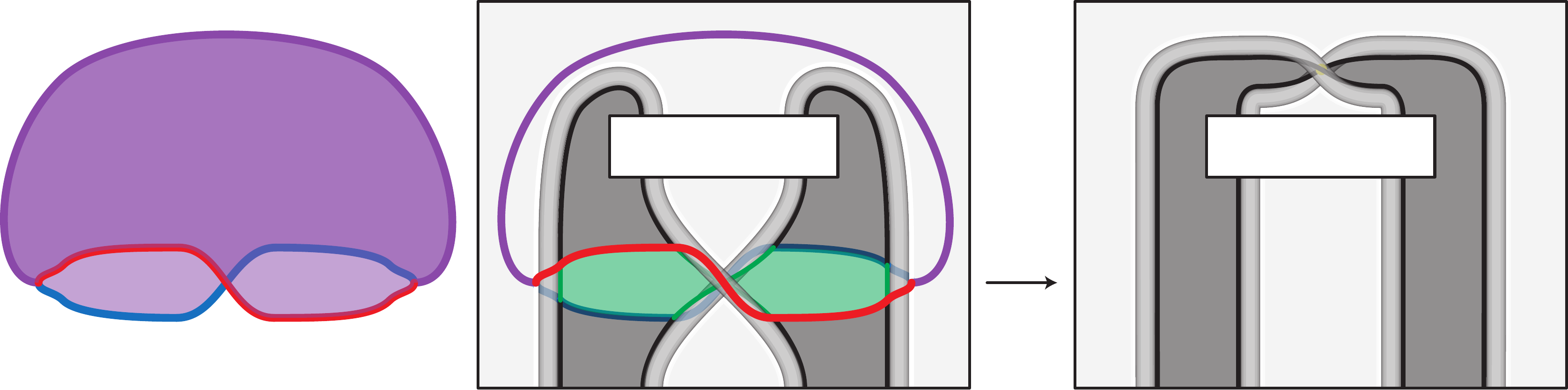}
\caption{A flyping cap and the associated flype move}\label{Fi:FlypeCap}
\end{center}
\end{figure}
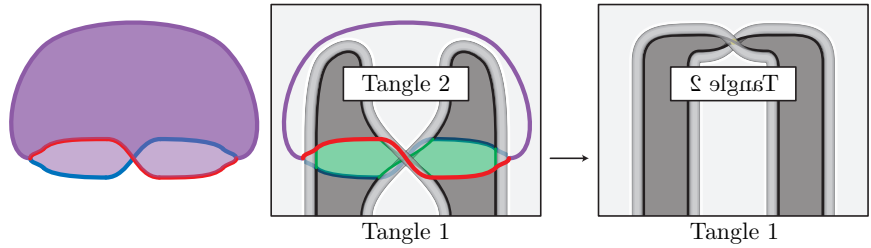

\section{The flyping theorem in thickened surfaces}
\label{S:MenascoH}

The arguments in \textsection\textsection3-5 and 7-8 of \cite{flyping} have been revised so that they apply directly in the context of this paper (with the obvious replacements of $S^3$ with $\Sigma\times I$, $S^2$ with $\Sigma$, essential with end-essential, and prime with locally prime
): $B$, $W$ are the checkerboard surfaces from a locally prime, cellular alternating diagram $D\subset \Sigma$ of a link $L\subset \Sigma\times I$, $F$ is an end-essential positive-definite surface spanning $L$, $v_F$ is comprised of the vertical arcs at the crossings where $F$ has crossing bands, and $D_{F,W}$ is the diagram determined via Theorem \ref{T:DBW} by $F,W$.  One then implements Menasco's crossing ball setup, isotopes $F$ into {\it fair position}, and performs a sequence of isotopy and re-plumbing moves according to a hierarchy: one only performs each move $k$ if $F$ is in $(k-1)$-good position, meaning that $F$ is in fair position and none of Moves 1 through $k-1$ are possible. See \cite{flyping} for the notations $C,v,\wh{W},S_\pm$ etc. associated with the crossing ball setup and for the precise definitions of fair position and Moves 1-10. Moves 1-9, all of which are isotopy moves, appear in Figure \ref{Fi:Moves}.  Move 10 is a re-plumbing move and is more complicated; see \cite{flyping}.

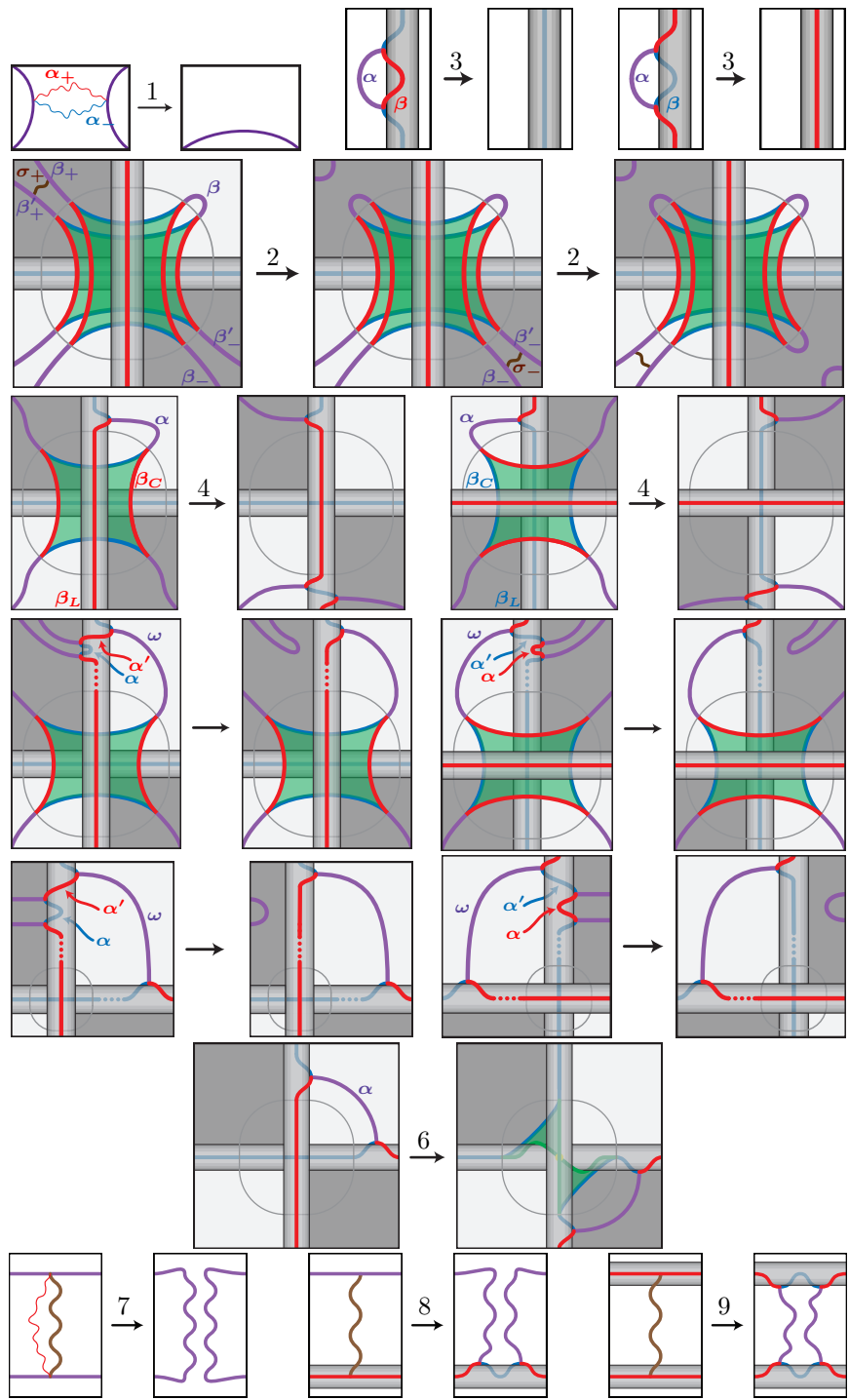
\begin{figure}
\begin{center}
\labellist
\small\hair 4pt
\pinlabel {{1}} [c] at 152 65
\tiny\hair 4pt
\pinlabel {$\red{\boldsymbol{\alpha_+}}$} [l] at 27 80
\pinlabel {$\Navy{\boldsymbol{\alpha_-}}$} [l] at 70 30
\endlabellist
\includegraphics[width=.35\textwidth]{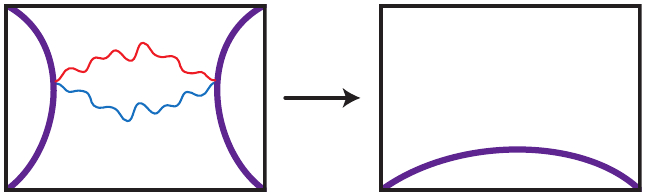}
\hfill
\labellist
\small\hair 4pt
\pinlabel {{3}} [c] at 71 58
\tiny\hair 4pt
\pinlabel {$\violet{\boldsymbol{\alpha}}$} [l] at 5 47
\pinlabel {$\red{\boldsymbol{\beta}}$} [l] at 25 30
\endlabellist
\includegraphics[width=.275\textwidth]{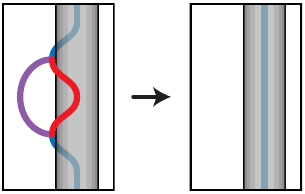}
\hfill
\labellist
\small\hair 4pt
\pinlabel {{3}} [c] at 71 58
\tiny\hair 4pt
\pinlabel {$\violet{\boldsymbol{\alpha}}$} [l] at 5 47
\pinlabel {$\Navy{\boldsymbol{\beta}}$} [l] at 25 30
\endlabellist
\includegraphics[width=.275\textwidth]{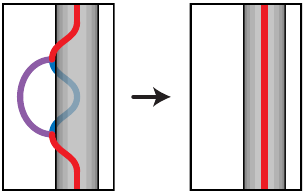}\\
\labellist
\small\hair 4pt
\pinlabel {{2}} [c] at 166 85
\pinlabel {{2}} [c] at 355 85
\tiny\hair 4pt
\pinlabel {$\violet{\boldsymbol{\beta}}$} [l] at 118 128
\pinlabel {$\violet{\boldsymbol{\beta_+}}$} [l] at 20 139
\pinlabel {$\violet{\boldsymbol{\beta'_+}}$} [l] at -2 115
\pinlabel {$\violet{\boldsymbol{\beta_-}}$} [l] at 102 8
\pinlabel {$\violet{\boldsymbol{\beta'_-}}$} [l] at 122 33
\pinlabel {$\sepia{\boldsymbol{\sigma_+}}$} [l] at -2 136
\pinlabel {$\sepia{\boldsymbol{\sigma_-}}$} [l] at 310 13
\pinlabel {$\violet{\boldsymbol{\beta_-}}$} [l] at 291 8
\pinlabel {$\violet{\boldsymbol{\beta'_-}}$} [l] at 311 33
\endlabellist
\includegraphics[width=\textwidth]{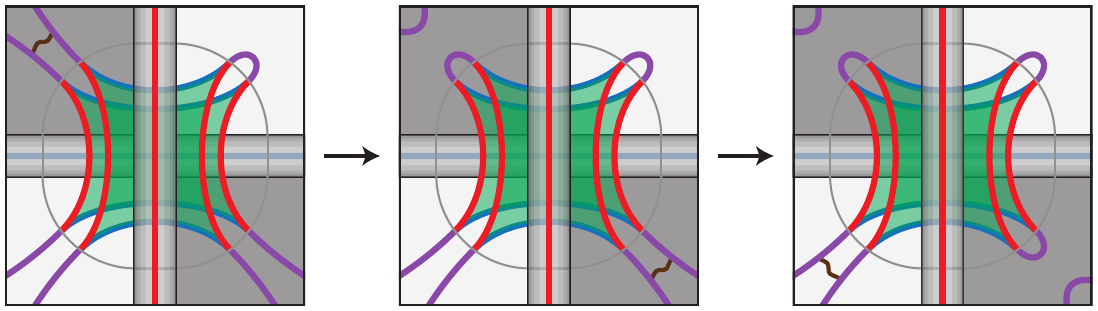}
\labellist
\small\hair 4pt
\pinlabel {{4}} [c] at 147 94
\tiny\hair 4pt
\pinlabel {$\violet{\boldsymbol{\alpha}}$} [l] at 102 145
\pinlabel {$\red{\boldsymbol{\beta_L}}$} [l] at 27 10
\pinlabel {$\red{\boldsymbol{\beta_C}}$} [l] at 87 100
\endlabellist
\includegraphics[width=.475\textwidth]{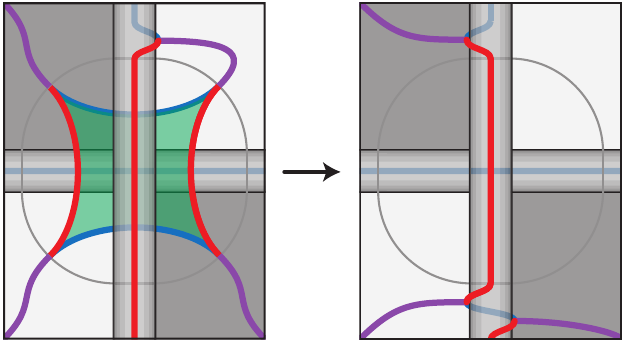}\hfill
\labellist
\small\hair 4pt
\pinlabel {{4}} [c] at 147 94
\tiny\hair 4pt
\pinlabel {$\violet{\boldsymbol{\alpha}}$} [l] at 0 147
\pinlabel {$\Navy{\boldsymbol{\beta_L}}$} [l] at 27 10
\pinlabel {$\Navy{\boldsymbol{\beta_C}}$} [l] at 5 100
\endlabellist
\includegraphics[width=.475\textwidth]{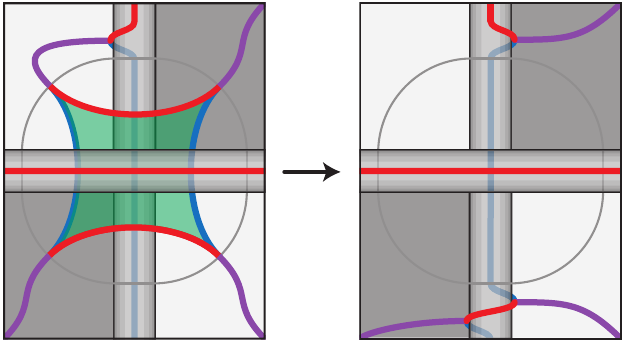}\\
\labellist
\tiny\hair 4pt
\pinlabel {$\violet{\boldsymbol{\omega}}$} [l] at 95 160
\pinlabel {$\Navy{\boldsymbol{\alpha}}$} [l] at 78 128
\pinlabel {$\red{\boldsymbol{\alpha'}}$} [l] at 81 140
\endlabellist
\includegraphics[width=.485\textwidth]{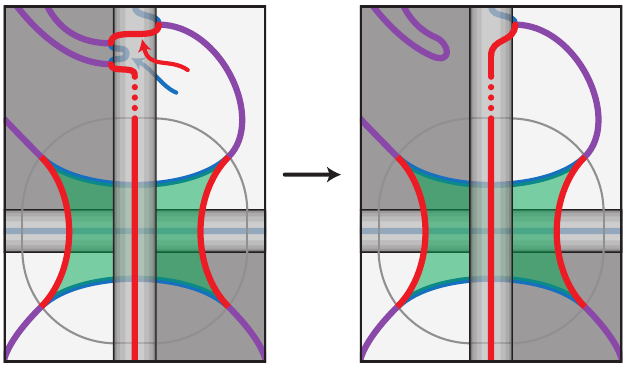}\hfill
\labellist
\tiny\hair 4pt
\pinlabel {$\violet{\boldsymbol{\omega}}$} [l] at 12 160
\pinlabel {$\red{\boldsymbol{\alpha}}$} [l] at 22 128
\pinlabel {$\Navy{\boldsymbol{\alpha'}}$} [l] at 14 140
\endlabellist
\includegraphics[width=.485\textwidth]{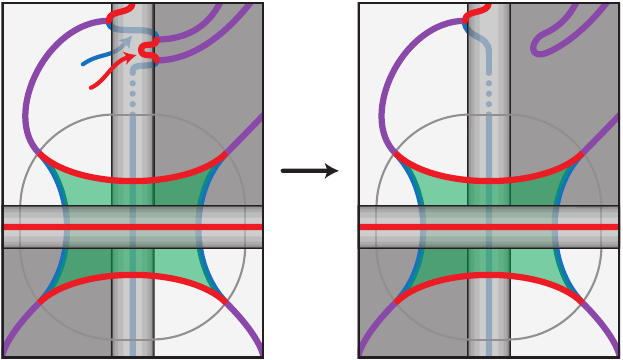}\\
\labellist
\tiny\hair 4pt
\pinlabel {$\violet{\boldsymbol{\omega}}$} [l] at 92 88
\pinlabel {$\red{\boldsymbol{\alpha'}}$} [l] at 60 95
\pinlabel {$\Navy{\boldsymbol{\alpha}}$} [l] at 55 70
\endlabellist
\includegraphics[width=.485\textwidth]{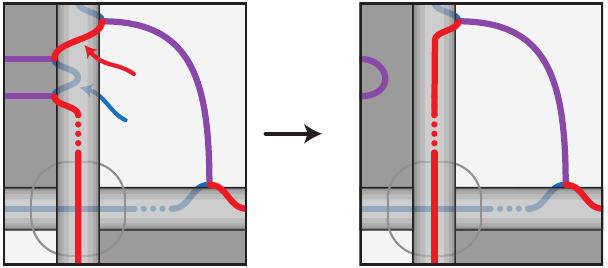}\hfill
\labellist
\tiny\hair 4pt
\pinlabel {$\violet{\boldsymbol{\omega}}$} [l] at 3 90
\pinlabel {$\Navy{\boldsymbol{\alpha'}}$} [l] at 36 94
\pinlabel {$\red{\boldsymbol{\alpha}}$} [l] at 38 72
\endlabellist
\includegraphics[width=.485\textwidth]{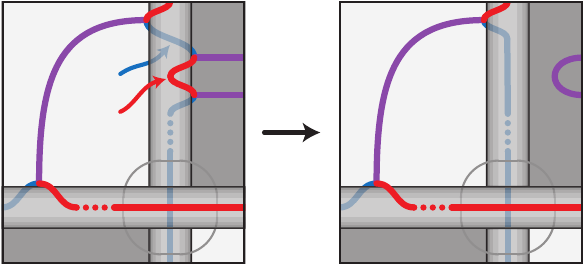}\\
\labellist
\small\hair 4pt
\pinlabel {{6}} [c] at 183 87
\tiny\hair 4pt
\pinlabel {$\violet{\boldsymbol{\alpha}}$} [l] at 123 127
\endlabellist
\includegraphics[height=.25\textwidth]{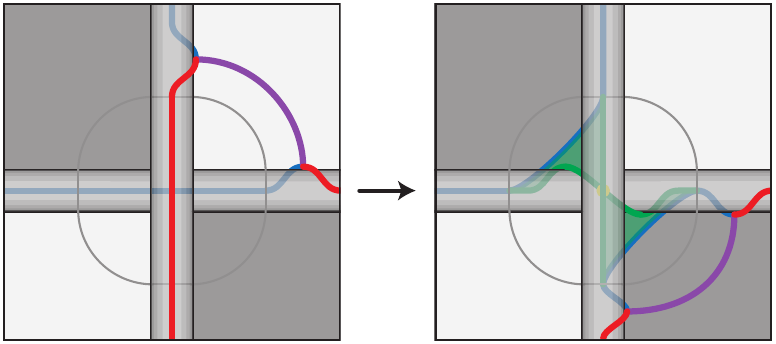}
\labellist
\small\hair 4pt
\pinlabel {{7}} [c] at 100 80
\pinlabel {{8}} [c] at 362 80
\pinlabel {{9}} [c] at 623 80
\endlabellist
\includegraphics[width=\textwidth]{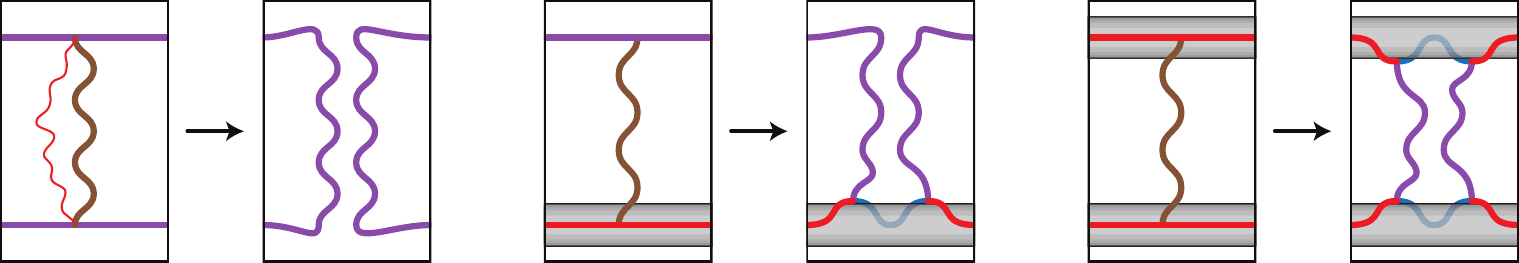}
\caption{Moves {1}-{9}}
\label{Fi:Moves}
\end{center}
\end{figure}

A few details are worth noting.  First, one must be more careful with push-through moves (see Definition 3.10 of \cite{flyping}) in thickened surfaces than in $S^3$.  The definition is the same (because it was written with this paper in mind!), but in addition to the three pictures shown top in Figure 19 of \cite{flyping}, three more pictures are possible.  See Figure \ref{Fi:Pt}.  In any case, if we wish to perform (or observe the possibility of) a push-through move along an arc $\alpha$ whose endpoints lie on a circle $\gamma$, we must now check that $\alpha$ is parallel in $S_+$ into $\gamma$; in \cite{flyping}, this was free.  Importantly, however, this is always the case.\footnote{In \cite{flyping}, see the definition of Move 2 and the proofs of Lemmas 3.22, 4.1, and 5.3 and of Propositions 8.2 and 8.3.}

\begin{figure}
\begin{center}
\labellist
\tiny\hair 4pt
\pinlabel {$\red{\boldsymbol{\beta}}$} [l] at 270 495
\pinlabel {$\brown{\boldsymbol{\alpha}}$} [l] at 220 435
\pinlabel {$\red{\boldsymbol{\beta}}$} [l] at 705 495
\pinlabel {$\brown{\boldsymbol{\alpha}}$} [l] at 620 435
\pinlabel {$\red{\boldsymbol{\beta}}$} [l] at 1140 505
\pinlabel {$\brown{\boldsymbol{\alpha}}$} [l] at 1020 435
\endlabellist
\includegraphics[width=.8\textwidth]{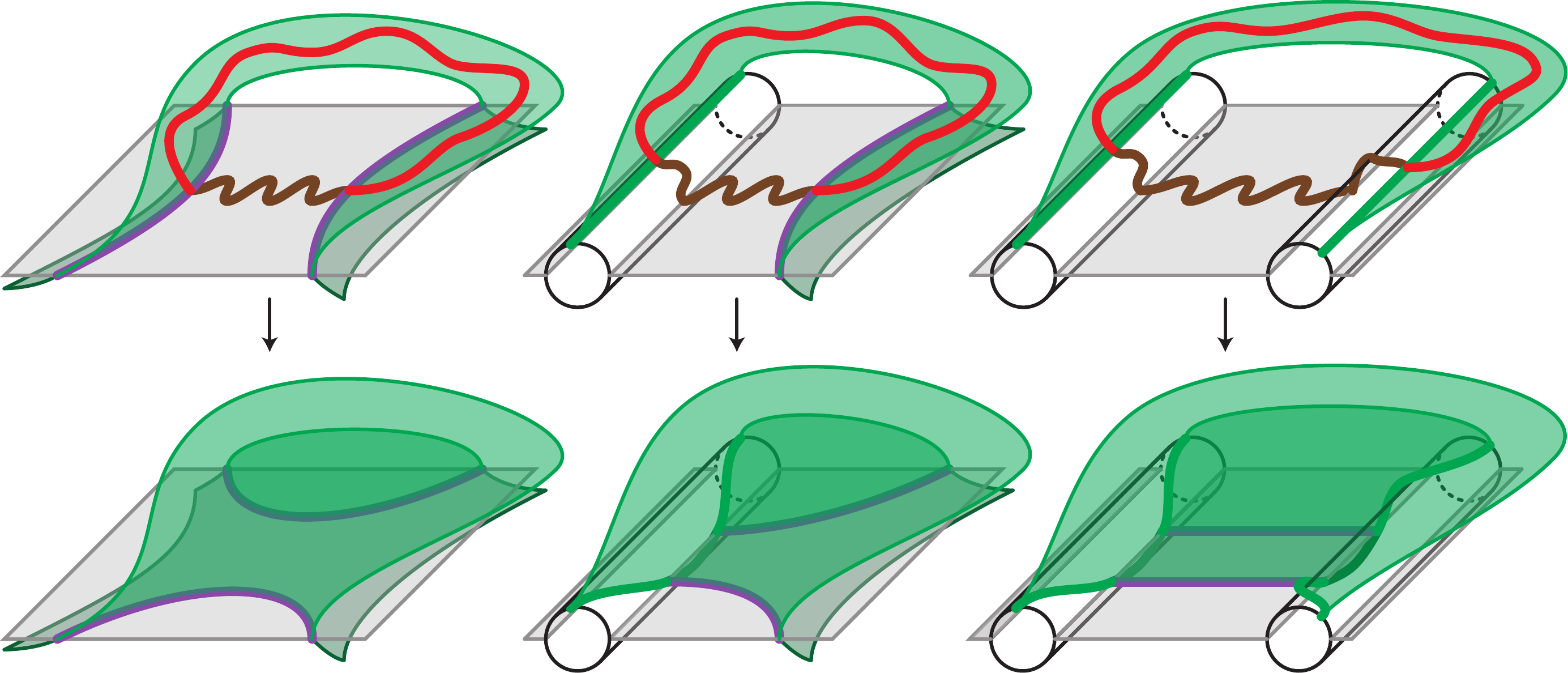}
\caption{Push-through moves in $\Sigma\times I$ need not appear as in Figure 19 of \cite{flyping}.}
\label{Fi:Pt}
\end{center}
\end{figure}

Second, whereas in \cite{flyping} every circle of $F\cap S_\pm$ was inessential in $S_\pm\approx S^2$, this property holds here only because the assumption that $F$ is end-incompressible allows us to require that $S_+\cup S_-$ cuts $F$ into disks (c.f. Definition 3.2 (h) and Lemma 3.3 of \cite{flyping}).  

Third, Sublemma 5.2 of \cite{flyping} implies there that the circles of $F\cap S_+$ are mutually nested, but this is less clear here. The proof of Lemma 5.3 of \cite{flyping} is thus written with this paper in mind, and is slightly more complicated as a result.

Adapting the arguments from \textsection\textsection3-5, 7-8 of \cite{flyping} thus gives:

\begin{theorem}%
\label{T:plumb}
If $D=D_{B,W}$ is a locally prime, cellular alternating diagram of $(\Sigma,L)$, then any end-essential, positive definite surface $F$ spanning $L$ is plumb-related to $B$
; likewise for end-essential negative-definite surfaces and $W$.\ari
\end{theorem}

\begin{cor}%
\label{C:tait}
With $F$ and $D$ as in Theorem \ref{T:plumb},
$\beta_1(B)=\beta_1(F)$ and $s(B)=s(F')$.\ari\footnote{This is also true if $L$ is non-stabilized and/or locally prime.} \end{cor}

\begin{theorem}[Part of Tait's extended first conjecture \cite{greene,kauff,mur,this,tur}]\label{T:tait1}
If $D,D'\subset\Sigma$ are alternating diagrams of a link $L\subset\Sigma\times I$, neither containing removable nugatory crossings, then $D$ and $D'$ have the same number of crossings.\footnote{When $D$ and $D'$ are locally {prime} and cellular alternating, Fact \ref{F:PGreene}, Theorems \ref{T:EndEss} and \ref{T:plumb}, and Corollary \ref{C:tait} immediately imply this. The general case then follows, as the number of crossings is additive under (de)stabilization, diagrammatic connect sum, and split union.}
\end{theorem}

\begin{theorem}
\label{T:Bad}
If $F$ is in 9-good position, then $F$ contains no saddle disks: $F\cap C=v_F$; hence, every circle $\gamma$ of $F\cap S_+$ is a flyping circle, and $D_{F,W}$ is related to $D$ by a sequence of flypes that preserve the isotopy class of $W$.\ari
\end{theorem}

\begin{theorem}[Tait's extended flyping conjecture]\label{T:mainthick}
All locally {prime}, cellular alternating diagrams $D=D_{B,W}$ and $D'=D_{B',W'}$ of the same link $L\subset \Sigma\times I$ are related by a sequence of flypes $D\to\cdots\to D''\to\cdots\to D'$ in which $D\to \cdots\to D''$ preserves the isotopy class of $W$ and $D''\to\cdots\to D'$ preserves the isotopy class of $B'$. 
\end{theorem}

Since writhe is invariant under flypes (recall Observation \ref{O:Flype}) and additive under diagrammatic connect sum and disjoint union, we obtain a new geometric proof of Tait's second conjecture:

\begin{theorem}[Tait's extended second conjecture \cite{bk18,bks19}]\label{T:tait2}
All locally {prime}, cellular alternating diagrams of a given link $L\subset \Sigma\times I$ have the same writhe.
\end{theorem}

Theorem \ref{T:mainthick} implies that, unlike a classical link and a link in $S^2\times I$, a link in a thickened surface of positive genus is not necessarily isotopic to the link obtained by reflecting horizontally (in the projection surface) and then vertically. More precisely, let $D\subset\Sigma$ be a locally {prime}, cellular alternating diagram of a link $L\subset \Sigma\times I$; let $\phi:\Sigma\to\Sigma$ be an orientation-reversing involution; let $D'\subset\Sigma$ be the diagram obtained from $\phi(D)$ by reversing all crossing information; and let $L'\subset \Sigma\times I$ be the link represented by $D'$. Note that $L'$ is the image of $L$ under the map $\Sigma\times I\to \Sigma\times I$ given by $(x,t)\mapsto (\phi(x),-t)$.

\begin{cor}\label{C:Entire}
With the setup above, if $D$ is locally {prime} and cellular alternating, then the links $L$ and $L'$ are isotopic in $\Sigma\times I$ if and only if the diagrams $D$ and $D'$ are flype-related on $\Sigma$. In particular, this is always true if $g(\Sigma)=0$, but not necessarily if $g(\Sigma)>0$. 
\end{cor}

\begin{example}\label{Ex:Entire}
The diagrams on $T^2$ shown right in Figure \ref{Fi:Entire} admit no non-trivial flypes and are non-isotopic; thus, by Corollary \ref{C:Entire}, they represent non-isotopic links in $T^2\times I$.
\end{example}

\section{Virtual links}\label{S:Virtual}

A {\it virtual link diagram} is the image of an immersion $\bigsqcup S^1\to S^2$ in which all self-intersections are transverse double-points, some of which are labeled with over-under information. These labeled points are called {\it classical crossings}, and the other double-points are called {\it virtual crossings}. Traditionally, virtual crossings are marked with a circle, as in Figure \ref{Fi:RMoves}.  
A {\it virtual link} is an equivalence class of such diagrams under generalized Reidemeister moves (R-moves), as shown in Figure \ref{Fi:RMoves}.  
There are seven types of such moves, the three {\it classical} moves and four {\it non-classical} moves.

\begin{figure}
\begin{center}
\includegraphics[width=\textwidth]{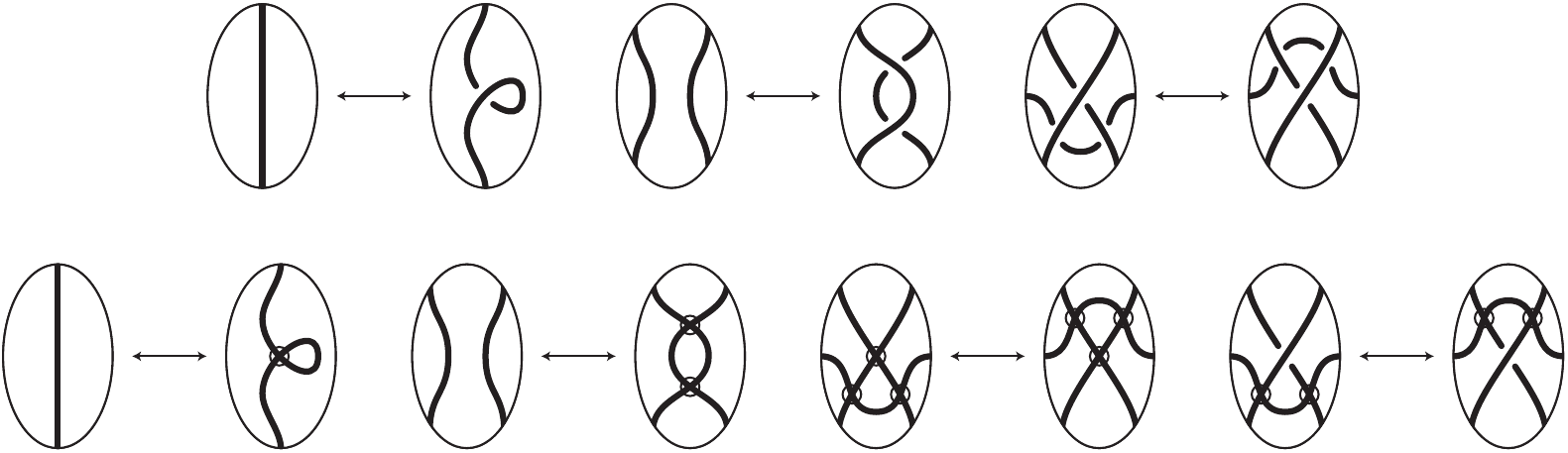}
\caption{Generalized Reidemeister moves}
\label{Fi:RMoves}
\end{center}
\end{figure}

\begin{notation}
Given a virtual link diagram $V\subset S^2$, let $[V]$ denote the set of all virtual diagrams related to $V$ by planar isotopy and non-classical R-moves.
\end{notation}

\begin{definition}
A virtual link diagram $V$ is {\bf nonsplit} if every $V'\in[V]$ is connected.
\end{definition}

\subsection{Correspondences}\label{S:Correspondence}

By work of Kauffman \cite{kauff98}, Kamada--Kamada \cite{kaka}, and Carter--Kamada--Saito \cite{cks02}, there is a bijective correspondence between virtual links and stable equivalence classes of links in thickened surfaces.  In \cite{primes}, the author introduces the following correspondence between equivalence classes of the associated {\it diagrams}:

\begin{figure}
\begin{center}
\includegraphics[width=\textwidth]{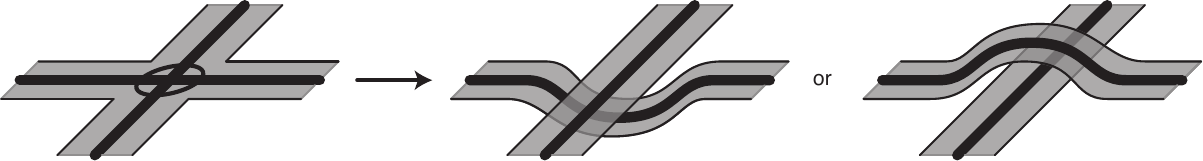}
\caption{Converting the neighborhood of a virtual link diagram to an abstract link diagram}
\label{Fi:VtoA}
\end{center}
\end{figure}

\begin{corr}\label{C:Diagrams}
The following gives a bijection from equivalence classes $[V]$ of nonsplit virtual link diagrams to cellular link diagrams on connected closed surfaces:

Choose $V\in[V]$, take a regular neighborhood $\nu V$ of $V$ in $S^2$, %
modify $\nu V$ near each virtual crossing of $V$ as shown in Figure \ref{Fi:VtoA},%
\footnote{At this intermediate stage, we have an {\it abstract link diagram}, which 
we will not need again.}
and cap off each boundary component (abstractly) with a disk.
\end{corr}

There is an important subtlety in Correspondences \ref{C:Diagrams} and \ref{C:Links}, pertaining to how one constructs a virtual link diagram from a given cellular pair $(\Sigma,D)$.  The idea of this construction is to embed $(\Sigma,D)$ into $S^3$ and project.  The subtlety is that, under this embedding, all crossings of $D$ must lie on the front of $\Sigma$.  See \textsection4.2 of \cite{primes} for details.

Correspondence \ref{C:Diagrams} gives a new diagrammatic perspective on a well-known correspondence \cite{kauff98,kaka,cks02}:

\begin{corr}\label{C:Links}
There is a correspondence between virtual links and stable equivalence classes of links in thickened surfaces: choose any representative diagram and apply the diagrammatic Correspondence \ref{C:Diagrams}.
\end{corr}

\subsection{Prime virtual link( diagram)s}\label{S:VPrime}

We adopt the following definitions from \cite{primes}:

\begin{definition}\label{D:VKPrime}
A diagrammatic connect sum decomposition $V=V_1\#V_2$ of a virtual link diagram is {\it nontrivial} if both $V_1$ and $V_2$ have classical crossings. 
A virtual link diagram $V$ is {\bf prime} if it has no nontrivial connect sum decomposition.
A virtual link $K\neq \bigcirc $ is {\bf prime} if for every diagram $V$ of $K$ and every decomposition $V=V_1\#V_2$, one diagram $V_i$ represents the classical unknot and the other represents $K$.
\end{definition}
 
Primeness of virtual links corresponds as follows to primeness of nonstabilized pairs $(\Sigma,L)$.  For simplicity, we add an assumption of checkerboard colorability, which does not appear in the full version of the theorem in \cite{primes}.

 \begin{theorem}[c.f. Theorem 5.7 of \cite{primes}]\label{T:01}
Given a nonsplit checkerboard colorable virtual link $K$ and the corresponding nonstabilized link $L$ in a thickened surface $\Sigma\times I$, 
\begin{enumerate}
\item $(\Sigma,L)$ is locally prime if and only if $K$ admits no nontrivial connect sum decomposition $K=K_1\#K_2$ in which $K_1$ is a classical link and $g(K_2)=g(K)$, and
\item $K$ is prime if and only if $(\Sigma,L)$ is pairwise prime.
\end{enumerate}
\end{theorem}









\subsection{Determining primeness of alternating virtual link diagrams}

Theorem \ref{T:ObviouslyPrimeEtc} states that an alternating pair $(\Sigma,L)$ is composite (in some sense) 
if and only if it is ``obviously so'' in a given reduced alternating diagram. To translate this result to alternating virtual link diagrams $V$, one must describe what it means for such $V$ to be ``obviously" composite (in either the local or pairwise sense). The solution in \cite{primes} is to use a new tool called {\it lassos} to capture the salient features of all diagrams in $[V]$ and find some suitable $V'\in[V]$.

A lasso for $V$ is a disk $X\subset S^2$ that contains all classical crossings of $V$ and no virtual ones, and $X$ is {\it acceptable} if the part of the diagram in $X$ is connected and the part of the diagram outside $X$ does not admit an ``obvious" simplification (see \cite{primes} for details). In \cite{primes}, the author proves (the proof of the first result is constructive):

\begin{prop}[Proposition 3.8 of \cite{primes}]\label{P:VLassoAcc}
Given a nonsplit virtual link diagram $V$ some $V'\in [V]$ admits an acceptable lasso.
\end{prop}

\begin{fact}[c.f. Proposition 7.3 of \cite{primes}]
Let $V\subset S^2$ be a nonsplit virtual link diagram corresponding to a cellular pair $(\Sigma,D)$.  Choose $V'\in[V]$ which admits an acceptable lasso $X$.  Given a crossing $c'\in V$ corresponding to a crossing $c\in D$, $c$ 
removably nugatory 
if and only if some disk $U\subset X$ has $\partial U\cap V'=\{c'\}$;
\end{fact}

In other words, $V'$, chosen constructively from $[V]$ so that it admits an acceptable lasso, has an ``obvious'' removable nugatory crossing whenever $(\Sigma,D)$ has a removable nugatory crossing. Moreover, $V'$ is (resp. locally) prime if and only if every $V''\in[V]$ is (resp. locally prime); see Observations 5.14-5.14 of \cite{primes}.  With this context, the culminating result of \cite{primes} is that, in the alternating case, one can determine by inspecting $V'$ whether or not it represents prime or locally prime virtual link:

\begin{theorem}[Theorem 7.8 of \cite{primes}]\label{T:VSplitEtc}
Let $V$ be an alternating 
diagram of a nonsplit virtual link $K$, and consider a diagram $V'\in[V]$ which 
admits an acceptable lasso $X$. Assume that $V'$ has at least one virtual crossing. Then:
\begin{enumerate}[label=(\arabic*)]
\item When $V'$ has no nugatory crossings, $K$ is prime if and only if $V'$ is prime.
\item When $V'$ has no removably nugatory crossings, $K$ is locally prime if and only if $V'$ is locally prime.
\end{enumerate}
\end{theorem}

\subsection{Tait's conjectures for virtual links}\label{S:VFlype}

\begin{definition}\label{D:VFlype}
{A (classical) {\bf flype} on a virtual link diagram appears as in Figure \ref{Fi:flype}, where $T_1$ contains no virtual crossings.}
\end{definition}

\begin{theorem}\label{T:virtual}
Any two locally {prime}, alternating diagrams of a given virtual link $\wt{L}$ are related by non-classical R-moves and (classical) flypes.%
\end{theorem}

\begin{proof}
Let $V$ and $V'$ be two such diagrams, and let $(\Sigma,D)$ and $(\Sigma',D')$ be the associated pairs under Correspondence \ref{C:Diagrams}.  By Kuperburg's theorem, we may identify $\Sigma\equiv \Sigma'$, and by Theorem \ref{T:mainthick}, there is a sequence of flype moves on $\Sigma$ taking $D$ to $D'$:
\[D=D_0\to D_1\to\cdots \to 
D_n=D'.\]  We will show for each $i=1,\hdots, n$ that there are virtual diagrams $V^2_{i-1}$ and $V^1_i$ which correspond to $(\Sigma,D_{i-1})$ and $(\Sigma,D_i)$ and which are related by a flype.  This will produce a sequence of virtual diagrams 
\begin{equation*}\label{E:VFlype}
V=V_0^1\to V_0^2\to V_1^1\to V_1^2\to\cdots\to 
\to V_n^1\to V_n^2=V'
\end{equation*}
where each $V_i^1\to V_i^2$ comes from a sequence of virtual R-moves and each $V_{i-1}^2\to V_i^1$ comes from a flype.  

Consider a flype $D_{i-1}\to D_i$; it is supported within a disk $X\subset\Sigma$.\footnote{That is, take $X$ to be the oval-shaped disk shown left in Figure \ref{Fi:flype}.}  Denote the quotient map $q:\Sigma\to \Sigma/X\equiv \Sigma$, and denote the underlying graph of $D_{i-1}$ by $G$.  Choose a spanning tree $T$ for the 4-valent graph $q(G)\subset\Sigma/X$, and take a regular neighborhood $\nu T$. Denote $U=q^{-1}(\nu T)$, and observe that $U$ is a disk in $\Sigma$ that contains $X$ and all crossings of $D_{i-1}$.\footnote{Thus, $U$ is a lasso for $(\Sigma,D_{i-1})$.}

Choose an embedding $\phi:\Sigma\to \wh{S^3}$ such that $\pi|_{\phi(U)}$ has no critical points and $\pi\circ\phi(U)\cap \pi\circ\phi(D\setminus U)=\varnothing$. Denote $f=\pi\circ\phi$ and $f(D_{i-1})=V^2_{i-1}$.  Observe that $f|_X$ is a homeomorphism onto its image, and so the disk $f(X)$ supports a flype $V^2_{i-1}\to V^1_i$ where $V^1_i$ corresponds to $(\Sigma,D_i)$.  

Thus, as needed, each $V_{i-1}^2\to V_i^1$ comes from a flype.  To complete the proof, we note that each $V_i^1\to V_i^2$ comes from a sequence of virtual R-moves, due to Correspondence \ref{C:Diagrams}, since both $V_i^1$ and $V_i^2$ correspond to the same cellularly embedded diagram $D_i$ on $\Sigma$.
\end{proof}

Since crossing number and writhe are invariant under flypes, we can also extend more parts of Tait's conjectures to virtual links:

\begin{theorem}\label{T:tait12v}
All locally {prime}, alternating diagrams of a given virtual link have the same crossing number and writhe.
\end{theorem}

\subsection{Non-uniqueness of minimal-genus connect sum}

Finally, as an additional corollary, we observe that, whereas minimal-genus connect sum is a well-defined operation for classical knots and for any classical knot with any virtual knot, minimal-genus connect sum is {\it not} a well-defined operation for virtual knots:

\begin{figure}
\begin{center}
\includegraphics[width=\textwidth]{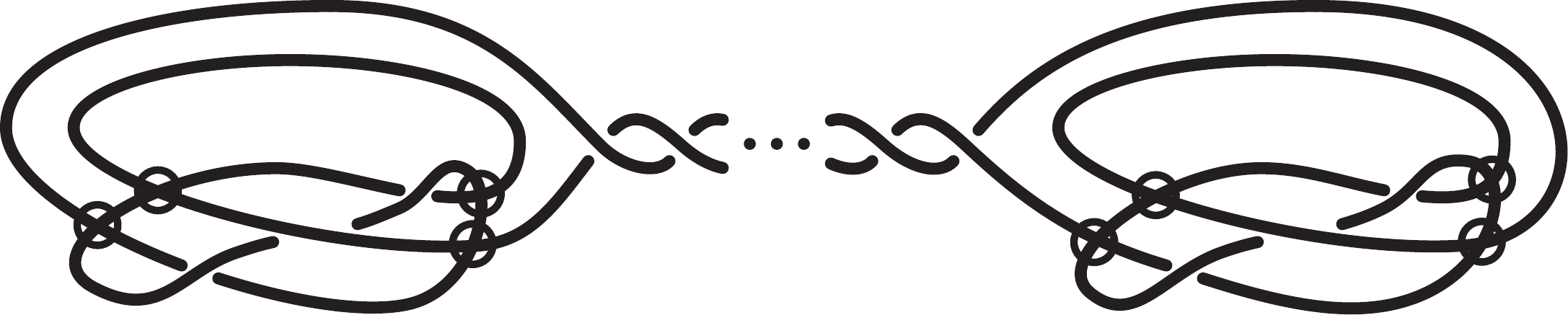}
\caption{There are infinitely many different ways to take minimal-genus connect sums of any two non-classical alternating links.}
\label{Fi:ConnectSumN}
\end{center}
\end{figure}

\begin{cor}\label{C:NonUCS}
Given any two non-classical, locally {prime}, alternating virtual links $K_1$ and $K_2$, there are infinitely many distinct virtual links $K$ with $g(K)=g(K_1)+g(K_2)$ that decompose as a connect sum of $K_1$ and $K_2$.
\end{cor}

This follows immediately from Theorem \ref{T:tait12v}, using the construction suggested in Figure \ref{Fi:ConnectSumN}.  
We conjecture that the same construction works more generally:

\begin{conjecture}
Given any two non-classical, locally {prime}, checkerboard colorable virtual links $K_1$ and $K_2$, there are infinitely many distinct virtual links $K$ with $g(K)=g(K_1)+g(K_2)$ that decompose as a connect sum of $K_1$ and $K_2$.
\end{conjecture}

\newpage

\section*{Appendix A: Cross-referencing with \cite{flyping}}

\begin{table}[h!]
\begin{center}
\begin{tabular}{|cc|cc|}
\hline
here& in \cite{flyping}&here& in \cite{flyping}\\ \hline
Prop. \ref{P:BdryParallel}&Prop. 2.5&
Prop. \ref{P:Ess}&Prop. 2.6\\
Obs. \ref{O:ArcW}& Fact 2.7&
Rem. \ref{R:CBEss} & Rem. 2.8\\
Def. \ref{D:Flype} & Def. 2.9&
Obs. \ref{O:Flype} & Obs. 2.10 
\\
Proc. \ref{Proc:ArcCollapse}&Proc. 2.23&
Proc. \ref{Proc:Kill1}& Proc. 2.24\\
Fact \ref{Fact:arccutdef}& Subl. 6.3 &
Fact \ref{F:AddTwist} & Subl. 6.4 \\
Fact \ref{F:Arc0} &Prop. 6.5&
Prop. \ref{P:Kill1} &Prop. 6.8\\
Prop. \ref{P:BdryParallel0} &Prop. 6.9&
Thm. \ref{T:DBW}&Thm. 2.35\\
Cor. \ref{C:DBW}&Cor. 2.36&
Thm. \ref{T:plumb} & Thm. 4.5\\%
Cor. \ref{C:tait}& Cor. 4.6&
Thm. \ref{T:Bad} &Thm. 5.4 \\
\hline
\end{tabular}
\caption{Cross-listing information with \cite{flyping}}
\end{center}
\end{table}

\end{document}